\documentclass{amsart}
\usepackage[a4paper, margin=1in]{geometry}
\usepackage{amssymb}
\usepackage[hyphens]{url}
\usepackage{amsmath, amsfonts}
\usepackage{amsthm}
\usepackage{mathrsfs}
\usepackage[title]{appendix}
\usepackage[utf8]{inputenc}
\usepackage[english]{babel}
\usepackage{cite}
\usepackage{hyperref}
\usepackage{cleveref}
\usepackage[section]{placeins}
\RequirePackage{textgreek}
\usepackage{enumitem}
\usepackage{bm}
\usepackage{nicefrac}
\usepackage{mathabx}
\usepackage[all,cmtip]{xy}
\usepackage{caption}
\usepackage{float}
\usepackage{mathtools}
\usepackage{xcolor}
\usepackage{xspace}
\usepackage{tikz}
\usetikzlibrary{arrows.meta}
\usepackage{lipsum}
\usepackage{wrapfig}

\hypersetup{
    colorlinks=true,
    linkcolor=blue,
    citecolor=cyan,
    urlcolor=red
}

\hypersetup{linktoc=all}
\numberwithin{equation}{subsection}

\newtheorem{thm}{Theorem}[section]
\newtheorem*{thm*}{Theorem}
\newtheorem{lem}[thm]{Lemma}
\newtheorem{prop}[thm]{Proposition}
\newtheorem{cor}[thm]{Corollary}

\newtheorem{defn}[thm]{Definition}

\newtheorem{que}[thm]{Question}
\newtheorem{fct}[thm]{Fact}
\newtheorem{claim}[thm]{Claim}
\newtheorem{mainthm}{Theorem}

\theoremstyle{remark}
\newtheorem{rem}[thm]{Remark}

\newcommand{\diag}{\textsubscript{\raisebox{0.4ex}{\scalebox{0.6}{\(\blacktriangle\)}}}}\xspace

\title{\scalebox{0.9}{Locally integrable cross sections and their intersection covolume}}

\author{Nachi Avraham-Re'em}
\address{Department of Mathematics, Chalmers and University of Gothenburg, Gothenburg, Sweden}
\curraddr{Department of Mathematics, Technion---Israel Institute of Technology, Haifa, Israel}
\email{nachi.avraham@gmail.com}

\author{Michael Bj\"{o}rklund}
\address{Department of Mathematics, Chalmers and University of Gothenburg, Gothenburg, Sweden}
\email{micbjo@chalmers.se}

\author{Rickard Cullman}
\address{Department of Mathematics, Chalmers and University of Gothenburg, Gothenburg, Sweden}
\email{cullman@chalmers.se}

\thanks{The research was supported by the Knut and Alice Wallenberg Foundation (KAW 2021.0258).}



\begin{document}

\begin{abstract}
We study systematically cross sections of probability preserving actions of unimodular groups and their associated transverse measures, and introduce the invariant \emph{intersection covolume} to quantify their periodicity. Our main theorem, derived from a higher order version of Kac's lemma, shows that the intersection covolume is bounded below by the intensity, with equality precisely when the action is induced by a lattice (in the sense of Mackey). We further prove that the natural cross sections of cut--and--project actions have finite intersection covolume.
\end{abstract}

\maketitle

\setcounter{tocdepth}{1}
\tableofcontents

\section{Introduction}

This paper, together with its companion paper \cite{AvBjCuII}, develops a systematic theory of intersection spaces associated with cross sections of probability preserving actions of locally compact second countable (lcsc) unimodular groups. Intersection spaces were originally introduced by the second author together with Hartnick and Karasik in \cite{bjorklund2025int}, where their construction relied on the additional assumption that the cross sections are separated. The present work removes this restriction and treats general locally integrable cross sections, thereby significantly enlarging the range of applications. In particular, we can treat random subsets of general, not necessarily uniform, approximate lattices (in particular those arising from cut--and--project schemes), extending the scope of the theory well beyond the framework accessible in \cite{bjorklund2025int}. In another direction, the framework developed in this work has been used by Hartnick and Sarti \cite{hartnick2025bounded} to establish bounded cohomological induction for transverse measured groupoids.

\smallskip

By considering general locally integrable cross sections, we unify the transverse measure construction for separated cross sections in ergodic theory with the Palm measure construction in point process theory. Accordingly, a key ingredient in our framework is a characterization of relative invariance in terms of the classical Mecke equation. This was known for lcsc abelian groups \cite{heveling2005,heveling2007}, and by proving it for general lcsc groups we confirm a conjecture of G.~Last~\cite[\S3.10]{last2009}, which is of independent interest already in point processes theory.

\smallskip

Intersection spaces were tailored to the study of random subsets of approximate lattices in the sense of \cite{BjHa2018}. Although the construction extends to arbitrary locally integrable cross sections, the canonically associated intersection space measure is typically infinite. Finiteness of the intersection space measure arises only in structurally constrained situations, most prominently in the presence of approximate lattice phenomena or after a suitable thinning. In these cases, its total mass, which we call the \emph{intersection covolume}, serves as a quantitative measure of periodicity of the system. This interpretation is made precise in Theorem~\ref{thm:mthm} below, whose proof is based on a new formula for the intersection covolume as a higher order version of Kac's lemma (see Section~\ref{sct:intvol}).

\subsection{Cross sections, intensity and transverse measures}

Cross sections provide a way to reduce the continuous dynamics of an lcsc unimodular group to a countable equivalence relation, while retaining key quantitative invariants. They appear naturally in point processes, ergodic theory, aperiodic order, and other related fields. Formally, let \(G\) be an lcsc unimodular group acting in a measure preserving way on a standard probability space \(\left(X,\mu\right)\). A {\bf cross section} is a Borel set \(Y\subset X\) meeting every \(G\)-orbit, such that each {\bf return times set}
\[Y_{x}=\{g\in G : g.x\in Y\}, \quad x\in X,\]
is locally finite. The {\bf return times set} of \(Y\) itself is the symmetric set
\[\Lambda_{Y}\coloneqq\{g\in G:g.Y\cap Y\neq\emptyset\}.\]

Every cross section carries a natural {\bf transverse measure} \(\mu_{Y}\), analogous to the Palm measure of point process theory, which will be introduced in detail in Section~\ref{sct:transverse}. Its total mass defines the {\bf intensity}
\[\iota_{\mu}\left(Y\right)\in\left(0,+\infty\right],\]
and \(Y\) is {\bf locally integrable} if \(\iota_{\mu}\left(Y\right)<+\infty\), in which case we speak of a {\bf transverse \(G\)-space} \(\left(X,\mu,Y\right)\).

\smallskip

A classical example comes from stationary point processes: the space \(G^{\ast}\) of locally finite configurations in \(G\), with \(G\) acting by right translations, has the natural cross section
\[G_{o}^{\ast}\coloneqq\{\omega\in G^{\ast}:e_{G}\in\omega\}.\]
The Palm measure construction (see~\cite{last2009} and the references therein) yields a correspondence between invariant measures on \(G^{\ast}\) and Palm measures on \(G_{o}^{\ast}\), and {\em finite intensity} of a point process means finiteness of the corresponding Palm measure.

\smallskip

The construction of transverse measures for flows (actions of \(\mathbb{R}\)) dates back to Ambrose--Kakutani (see~\cite[\S12]{nadkarni2013}). For unimodular groups, separated cross sections (those with \(\Lambda_{Y}\) discrete) have been studied in~\cite{avni2010entropy, kyed2015, Slutsky2017, bjorklund2025int} among others. In order to extend the intersection spaces construction due to \cite{bjorklund2025int} to general cross sections, we need to extend the transverse measure construction accordingly. To this end we prove that the Mecke equation characterizes relative invariance of transverse measures, thereby confirming Last's conjecture~\cite[\S3.10]{last2009}, and this leads to a canonical {\em transverse correspondence}:
\[\bigl\{\text{\(G\)-invariant \(\sigma\)-finite measures on } X\bigr\}\ \longleftrightarrow\ \bigl\{\text{relatively \(G\)-invariant \(\sigma\)-finite measures on }Y\bigr\}.\]

\subsection{Periodicity and intersection covolume}

We now consider the following (minimal) notion of periodicity for cross sections: a cross section \(Y\) is called {\bf completely periodic} if each return times set \(Y_{x}\) is a coset of a fixed lattice \(\Gamma<G\). In this case, the action factors through the homogeneous space \(\left(\Gamma\backslash G,m_{\Gamma\backslash G}\right)\), and \(Y\) collapses to a single point in the factor. More generally, for a transverse \(G\)-space \(\left(X,\mu,Y\right)\), we seek to quantify the deviation of \(Y\) from complete periodicity. To this end, we use the \emph{intersection space} construction presented in~\cite{bjorklund2025int}, adapted to general cross sections using the aforementioned transverse correspondence.

\smallskip

Let then \(\left(X,\mu\right)\) be probability preserving \(G\)-space with a locally integrable cross section \(Y\), and let \(\mu_{Y}\) be its transverse measure. The associated {\bf intersection space} (of order \(2\)) is the measure preserving \(G\)-space \(\left(Y^{\left[2\right]},\mu^{\left[2\right]}\right)\) defined as follows. The underlying Borel \(G\)-space is
\[Y^{\left[2\right]}:=\{\left(g.y_{1},g.y_{2}\right):g\in G,y_{1},y_{2}\in Y\},\]
with the diagonal action of \(G\). One notes that \(Y\times Y\) serves as a cross section of \(Y^{\left[2\right]}\), and therefore, using the transverse correspondence, the \(\sigma\)-finite measure \(\mu^{\left[2\right]}\) is the unique measure whose transverse measure with respect to \(Y\times Y\) is the product transverse measure \(\mu_{Y}\otimes\mu_{Y}\). We then define the {\bf intersection covolume} (of order \(2\)) of \(\left(X,\mu,Y\right)\) to be the quantity
\[I_{\mu}\left(Y\right):=\mu^{\left[2\right]}\big(Y^{\left[2\right]}\big)\in\left(0,+\infty\right].\]
This will be presented in details in Section~\ref{sct:intvol}, where we define intersection covolume \(I_{\mu}^{r}\left(Y\right)\) for all \(r\in\mathbb{N}\).

\smallskip

Recall the well-known notion of a \(G\){\bf-factor} or \(G\){\bf-factor map} between probability preserving \(G\)-spaces. For transverse \(G\)-spaces, this notion must be refined to account for the additional structure. We thus call a map \(\phi:\left(X,\mu,Y\right)\to\left(W,\nu,Z\right)\) a {\bf transverse} \(G\){\bf-factor} if \(\phi:\left(X,\mu\right)\to\left(W,\nu\right)\) is a \(G\)-factor map of probability preserving \(G\)-spaces, such that
\[Y_{x}=Z_{\phi\left(x\right)}\text{ for }\mu\text{-a.e. }x\in X.\]
In Section~\ref{sct:transfact}, we characterize transverse \(G\)-factors in terms of transverse measures. Importantly, both intensity and intersection covolume are preserved under transverse \(G\)-factors (Propositions~\ref{prop:factor} and~\ref{prop:interfactor}).

\begin{mainthm}\label{thm:mthm}
For every probability preserving \(G\)-space \(\left(X,\mu\right)\) with a locally integrable cross section \(Y\),
\[I_{\mu}\left(Y\right)\geq\iota_{\mu}\left(Y\right),\]
with equality if and only if there is a transverse \(G\)-factor from \(\left(X,\mu,Y\right)\) onto a homogeneous space \(\left(\Gamma\backslash G,m_{\Gamma\backslash G},\{\Gamma\}\right)\) for a lattice \(\Gamma<G\).
\end{mainthm}

The construction of \emph{induced actions} from lattices was introduced by Mackey, and we refer to~\cite[\S II]{zimmer1978induced} for background. Then, from Theorem~\ref{thm:mthm} and a theorem of Zimmer~\cite[Theorem~2.5]{zimmer1978induced}, we obtain:

\begin{cor}
An ergodic probability preserving \(G\)-space \(\left(X,\mu\right)\) is an \emph{induced action} from a lattice in \(G\) if and only if it admits a cross section \(Y\) satisfying \(I_{\mu}\left(Y\right)=\iota_{\mu}\left(Y\right)\).
\end{cor}

The intersection covolume \(I_{\mu}\left(Y\right)\) can be infinite even when \(\iota_{\mu}\left(Y\right)<+\infty\). For instance, the Poisson point process on \(\mathbb{R}\) has finite intensity, while its intersection covolume can be shown to be infinite. It is therefore of particular interest to study situations where the inequality in Theorem~\ref{thm:mthm} is strict, yet the intersection covolume remains finite, thus indicating a degree of periodicity without being completely periodic. With a simple use of suspension flows, we will show in Section~\ref{sct:nogap} that for every \(\epsilon>0\), there is an ergodic transverse \(\mathbb{R}\)-space \(\left(X,\mu,Y\right)\) such that \(\iota_{\mu}\left(Y\right)<I_{\mu}\left(Y\right)<\left(1+\epsilon\right)\cdot\iota_{\mu}\left(Y\right)\).

\subsection{Cut--and--project actions and their cross sections}

Cut--and--project schemes arise in the study of quasicrystals and aperiodic order, where \emph{cut--and--project sets} provide ordered but non-periodic point patterns in groups. For uniform cut--and--project schemes (when the underlying lattice is uniform), it was proved in~\cite[Theorem~1.4]{bjorklund2025int} that the intersection measure is finite. We extend this to nonuniform lattices, and thus their intersection spaces constitute new examples of cross sections with finite intersection covolume but which are far from completely periodic.

A {\bf cut--and--project scheme} is a triple \(\left(G,H;\Gamma\right)\), where \(G\) and \(H\) are unimodular lcsc groups, and
\(\Gamma<G\times H\) is a lattice that projects injectively to \(G\) and densely to \(H\). Consider the homogeneous space
\[(\Xi,\xi)\coloneqq\bigl(\Gamma\backslash\left(G\times H\right),m_{\Gamma\backslash\left(G\times H\right)}\bigr),\]
as a Borel \(G\)-space (\(G\) acts on the first coordinate). A measurable set \(W\subset H\) is called a {\bf window}, if it is relatively compact, has nonempty interior, and it is aperiodic in the sense that no nontrivial element of \(H\) stabilizes \(W\). For a window \(W\subset H\), define a cross section of \(\Xi\) by
\[Y_{W}\coloneqq\{\Gamma\left(e_{G},w\right):w\in W\}.\]
The return times set of \(Y_{W}\) is a \emph{cut--and--project set} in \(G\):
\[\Lambda_{Y_{W}}=\operatorname{proj}_{G}\left(\Gamma\cap\left(G\times WW^{-1}\right)\right).\]

The following theorem generalizes
\cite[Theorem~1.4]{bjorklund2025int} to general cut--and--project schemes.

\begin{mainthm}\label{mthm:finintcov}
For every cut--and--project scheme \((G,H;\Gamma)\) and every window \(W\subset H\),
\[\iota_{\xi}\left(Y_{W}\right)<I_{\xi}\left(Y_{W}\right)<+\infty.\]
\end{mainthm}

In our companion paper~\cite{AvBjCuII}, we deal with cut--and--project schemes \(\left(G,H;\Gamma\right)\) for abelian groups \(G\) and \(H\), and in this case we compute explicit formulas for \(\iota_{\xi}\left(Y_{W}\right)\) and \(I_{\xi}\left(Y_{W}\right)\).

\subsection{Groups without lattices}

When the group \(G\) does not admit a lattice, Theorem~\ref{thm:mthm} forces a strict inequality
\(I_{\mu}\left(Y\right)>\iota_{\mu}\left(Y\right)\) for every cross section. It is therefore a natural problem to find good lower bounds for the ratio \(I_{\mu}\left(Y\right)/\iota_{\mu}\left(Y\right)\), and to identify actions that achieve these bounds. In our companion paper~\cite{AvBjCuII}, we give some answers to this question for a large class of abelian groups, including the \(p\)-adic group \(\mathbb{Q}_{p}\) (\(p\) prime) and the group \(\mathbb{A}_{\mathrm{fin}}\) of finite adeles, and show in particular the following:

\begin{thm*}
Let \(G=\mathbb{Q}_{p}\) or \(G=\mathbb{A}_{\mathrm{fin}}\). For every ergodic probability preserving \(G\)-space \(\left(X,\mu\right)\) with a cross section \(Y\) such that \(\Lambda_{Y}\) is uniformly discrete, it holds that
\[I_{\mu}\left(Y\right)\geq 2\cdot\iota_{\mu}\left(Y\right),\]
with equality if and only if there is a transverse \(G\)-factor from \(\left(X,\mu,Y\right)\) onto a cut--and--project space coming from a scheme \(\left(G,\mathbb{R};\Gamma\right)\) equipped with the cross section \(Y_{W}\) for some compact interval \(W\subset\mathbb{R}\).
\end{thm*}

A class of groups without lattices introduced in~\cite[\S2.4]{BjHa2018} is the one-parameter family of nilpotent Lie groups \(G_{\lambda}\), \(\lambda\in\left(0,+\infty\right)\). It was shown there that \(G_{\lambda}\) admits a lattice if and only if \(\lambda\) is rational, and that \(G_{\lambda}\) admits cut--and--project actions when \(\lambda\) is an algebraic irrational. It was shown in~\cite[Theorem~1.5]{Machado2020} that \(G_{\lambda}\) admits no cut--and--project action when \(\lambda\) is transcendental. This motivates the following:

\begin{que}
For the nilpotent Lie group \(G_{\lambda}\):
\begin{enumerate}
    \item When \(\lambda\) is an algebraic irrational, do cut--and--project \(G_{\lambda}\)-spaces minimize \(I_{\mu}\left(Y\right)\) among all transverse \(G_{\lambda}\)-spaces?
    \item When \(\lambda\) is transcendental, does there exist a transverse \(G_{\lambda}\)-space \(\left(X,\mu,Y\right)\) with \(I_{\mu}\left(Y\right)<+\infty\)?
\end{enumerate}
\end{que}

\section{General preliminaries}

Throughout this work, \(G\) stands for a locally compact, second countable (lcsc) unimodular group, and \(m_{G}\) stands for a fixed Haar measure of \(G\). Let \(X\) be a standard Borel space. That is, \(X\) is equipped with a \(\sigma\)-algebra which is the Borel \(\sigma\)-algebra of some Polish topology (completely metrizable and separable) on \(X\). We will call \(X\) a {\bf Borel \(G\)-space} when \(G\) acts on \(X\) in a Borel fashion, namely the action map \(\mathrm{a}:G\times X\to X\) is jointly measurable. As customary, we will abbreviate the action map by \(\left(g,x\right)\mapsto g.x\), and relate to each group element \(g\) as an actual Borel automorphism of \(X\). A {\bf cross section}\footnote{Some sources use the term \emph{transversal} to refer to what is here called a cross section. However, in contemporary literature, a \emph{transversal} usually refers to a set that intersects each orbit exactly once.} of a Borel \(G\)-space \(X\) is a Borel set \(Y\) such that
\[Y_{x}\coloneqq \{g\in G: g.x\in Y\}\]
is a nonempty and locally finite for every \(x\in X\). In particular, \(Y\) intersects every orbit in \(X\), thus \(G.Y=X\). Importantly, since \(Y_{x}\) contains a coset of the stabilizer of \(x\), a necessary condition to admitting a cross section is that all stabilizers are at most countable. We note the identity
\[Y_{g.x}=Y_{x}g^{-1},\quad g\in G,\,\,x\in X.\]

The reader should keep in mind the fundamental setting of point processes on \(G\), mentioned in the introduction, which is a (very important) particular case of our broader setting. Let \(G^{\ast}\) be the space of locally finite configuration of points in \(G\), with the \(\sigma\)-algebra generated by the maps
\[N_{K}:G^{\ast}\to\mathbb{Z}_{\geq 0}\cup\{+\infty\},\quad N_{K}\left(\omega\right)=\left|\omega\cap K\right|,\]
for all relatively compact Borel sets \(K\subset G\). Then \(G^{\ast}\) becomes a Borel \(G\)-space with the action \(g.\omega=\omega g^{-1}=\{xg^{-1}: x\in\omega\}\), and it has the natural cross section
\[G_{o}^{\ast}\coloneqq \{\omega\in G^{\ast}: e_{G}\in\omega\}.\]
Thus, when mentioning the setting of \emph{point processes}, it is in our language nothing but the Borel \(G\)-space \(G^{\ast}\) with its canonical cross section \(G_{o}^{\ast}\). Observe that for every Borel \(G\)-space \(X\) with a cross section \(Y\), the map \(X\to G^{\ast}\), \(x\mapsto Y_{x}\), is \(G\)-equivariant and maps \(Y\) to \(G_{o}^{\ast}\).

For a standard Borel space \(X\), we denote by \(\mathcal{M}\left(X\right)\) the set of Borel \(\sigma\)-finite measures on \(X\). For \(\mu\in\mathcal{M}\left(X\right)\) and a Borel function \(f:X\to\left[0,+\infty\right]\) we write
\[\mu\left(f\right)\coloneqq \int_{X}f\left(x\right)d\mu\left(x\right).\]
It is convenient for us, all along this work, to work with test functions taking values in \(\left[0,+\infty\right]\). The main reason is that this class of test functions is closed to periodizations (see below).

For a Borel \(G\)-space \(X\), we denote by
\[\mathcal{M}^{G}\left(X\right)\subset\mathcal{M}\left(X\right)\]
the set of measures in \(\mathcal{M}\left(X\right)\) which are \(G\)-invariant. When \(\mu\in\mathcal{M}^{G}\left(X\right)\), we will call \(\left(X,\mu\right)\) a {\bf measure preserving \(G\)-space}. A measure preserving \(G\)-space \(\left(X,\mu\right)\) is {\bf ergodic} if every \(G\)-invariant set in \(X\) is either \(\mu\)-null or \(\mu\)-conull.

\section{Mecke equation and relative invariance}\label{sct:mecke}

\subsection{Background}

Here we establish a general result about classes of measures defined on cross sections, which will be important to the upcoming transverse correspondence~\ref{thm:corresp}. In point processes, Palm measures admit a classical intrinsic characterization known as the \emph{Mecke equation}~\cite{mecke1967}. The relations between the Mecke equation and various notions of invariance are well-studied in the point processes literature, and one such property was introduced by Thorisson under the name \emph{point-stationarity} (see~\cite{last2009} and the references therein). The fact that Palm measures are characterized by point-stationarity is classical for point processes on \(\mathbb{R}\) and it was generalized by Heveling and Last, first for point processes on \(\mathbb{R}^{d}\)~\cite[Theorem 4.1]{heveling2005}, and then for point processes on lcsc abelian groups~\cite{heveling2007} (see also~\cite[Theorem 3.44]{last2009} and~\cite[\S6]{baccelli2024}). For general lcsc groups this characterization was conjectured by G.~Last~\cite[\S3.10]{last2009}.

In the following we formulate the Mecke equation and confirm this in the broader setting of general cross sections and arbitrary lcsc groups (not only unimodular). Specifically, we show that a measure on a cross section satisfies the Mecke equation if and only if it is invariant under the countable Borel equivalence relation induced by the group orbits on the cross section. This fundamental notion of measure invariance in ergodic theory, is precisely point-stationarity in the terminology of point processes. Our proof is simplified by invoking the Lusin--Novikov uniformization theorem several times.

\subsection{Basic definitions}

Let us start by defining measure invariance to countable Borel equivalence relations. For the basics of this notion of measure invariance see~\cite[\S8]{kechris2004topics},~\cite[\S4.2 \& 4.7]{Kechris2024}. Let \(A\) be a standard Borel space and \(E\) a countable Borel equivalence relation on \(A\). Thus, \(E\) is a Borel subset of \(A\times A\), the relation \(\left(a,a^{\prime}\right)\in E\) for \(a,a^{\prime}\in A\) is an equivalence relation, and each \(E\)-class \(E\left(a\right)\coloneqq\{a^{\prime}\in A:\left(a,a^{\prime}\right)\in E\}\) is at most countable. We associate with \(E\) the set (pseudo-group)  \(\bigl[\!\bigl[E\bigr]\!\bigr]\) of all partial transformations
\[\tau:\mathrm{Dom}\left(\tau\right)\longrightarrow\mathrm{Rng}\left(\tau\right)\text{ for some Borel sets }\mathrm{Dom}\left(\tau\right),\mathrm{Rng}\left(\tau\right)\text{ in }A,\]
such that \(\tau\) is a Borel bijection of \(\mathrm{Dom}\left(\tau\right)\) onto \(\mathrm{Rng}\left(\tau\right)\), and
\[\mathrm{Graph}\left(\tau\right)\coloneqq\{\left(a,\tau\left(a\right)\right):a\in\mathrm{Dom}\left(\tau\right)\}\subseteq E.\]

\begin{defn}\label{dfn:invcber}
Let \(E\) be a countable Borel equivalence relation on a standard Borel space \(A\). Define
\[\mathcal{M}^{E}\left(A\right)\subset\mathcal{M}\left(A\right)\]
to be the set of all measures \(\nu\in\mathcal{M}\left(A\right)\) that satisfy
\[\nu\left(\mathrm{Dom}\left(\tau\right)\right)=\nu\left(\mathrm{Rng}\left(\tau\right)\right)\text{ for every }\tau\in\bigl[\!\bigl[E\bigr]\!\bigr].\]
\end{defn}

For a Borel \(G\)-space with a cross section \(Y\), define the countable Borel equivalence relation
\[E_{G}^{Y}\coloneqq \left\{\left(y,y^{\prime}\right)\in Y\times Y: G.y=G.y^{\prime}\right\}\subseteq Y\times Y.\]
Thus, \(E_{G}^{Y}\) is the restriction to \(Y\) of the orbit equivalence relation induced by the action of \(G\) on \(X\). Since \(Y\) is a cross section, \(E_{G}^{Y}\) is indeed countable, namely each \(E_{G}^{Y}\)-class is countable.

\begin{rem}
The notion of \emph{point-stationarity}, that is defined as invariance under \emph{point-shifts} in the language of point processes, corresponds precisely to this notion of measure invariance to countable Borel equivalence relations. The \emph{point-shifts} are precisely the elements of \(\bigl[\!\bigl[E_{G}^{Y}\bigr]\!\bigr]\), where \(E_{G}^{Y}\) is the Borel equivalence relation obtained by restricting the (potentially uncountable) orbit equivalence relation \(G\) induces on \(G^{\ast}\) to the canonical cross section \(Y = G_{o}^{\ast}\), consisting of point configurations that contain the identity element. See~\cite[\S3.1]{heveling2005} and the references therein.
\end{rem}

We now formulate the Mecke equation for a Borel \(G\)-space \(X\) with a cross section \(Y\). To this end, define first two types of \emph{periodizations} of a Borel function \(f:G\times Y\to\left[0,+\infty\right]\):
\begin{itemize}
    \item The \(X\){\bf-periodization} of \(f\) is the Borel function
    \[f_{X}:X\longrightarrow\left[0,+\infty\right],\quad f_{X}\left(x\right)\coloneqq\sum\nolimits_{g\in Y_{x}}f\left(g^{-1},g.x\right).\]
    \item The \(Y\){\bf-periodization} of \(f\) is the Borel function
    \[f_{Y}:Y\longrightarrow\left[0,+\infty\right],\quad f_{Y}\left(y\right)\coloneqq \sum\nolimits_{g\in Y_{y}}f\left(g,y\right).\]
\end{itemize}

The Mecke equation relates the \(X\)-periodization and the \(Y\)-periodization of every \(f\) as follows.

\begin{defn}
Let \(X\) be a Borel \(G\)-space with a cross section \(Y\). Define
\[\mathcal{M}_{\mathrm{Mecke}}^{G}\left(Y\right)\subset\mathcal{M}\left(Y\right)\]
to be the set of measures \(\nu\in\mathcal{M}\left(Y\right)\) satisfying the {\bf Mecke equation}(s):
\begin{equation}\label{eq:mecke}
\nu\left(f_{X}\right)=\nu\left(f_{Y}\right)\text{ for every Borel function }f:G\times Y\to\left[0,+\infty\right].
\end{equation}
\end{defn}

We can now formulate the main result of this section, which was conjectured in~\cite[\S3.10]{last2009}.

\begin{thm}\label{thm:relativemecke}
Let \(G\) be an lcsc group. Then for every Borel \(G\)-space \(X\) with a cross section \(Y\),
\[\mathcal{M}_{\mathrm{Mecke}}^{G}\left(Y\right)=\mathcal{M}^{E_{G}^{Y}}\left(Y\right).\]
Thus, a measure in \(\mathcal{M}\left(Y\right)\) satisfies the Mecke equation \eqref{eq:mecke} if and only if it is \(E_{G}^{Y}\)-invariant.
\end{thm}

Our proof of Theorem~\ref{thm:relativemecke} uses crucially the Lusin--Novikov uniformization theorem. We use the following form of this classical theorem that appears in~\cite[Theorem (18.10)]{kechris2012classical}.\footnote{Note that by the first part of~\cite[Theorem (18.10)]{kechris2012classical}, every \emph{Borel graph} is an actual graph of a Borel map. Then the following is a reformulation of the \emph{Moreover} part of~\cite[Theorem (18.10)]{kechris2012classical}.}

\begin{thm}[Lusin--Novikov]\label{thm:lusinnovi}
Let \(A,B\) be standard Borel spaces and let \(P\subseteq A\times B\) be a Borel set such that \(P_{a}\coloneqq\left\{b\in B:\left(a,b\right)\in P\right\}\) is at most countable for every \(a\in A\). Then \(D\coloneqq\operatorname{proj}_{A}\left(P\right)\) is a Borel set, and there exist Borel maps \(\lambda_{n}:D\to B, n\in\mathbb{N}\), such that \(P_{a}=\{\lambda_{n}\left(a\right):n\in\mathbb{N}\}\) for every \(a\in D\).
\end{thm}

The first way we rely on the Lusin--Novikov theorem~\ref{thm:lusinnovi} is the following well-known reformulation of the general notion of invariance of measures to countable Borel equivalence relations. For a proof (relying on the Lusin--Novikov theorem)  see~\cite[\S8]{kechris2004topics}.

\begin{fct}\label{fct:mtp}
Let \(A\) be a standard Borel space and \(E\) a countable Borel equivalence relation on \(A\). Then every \(\nu\in\mathcal{M}^{E}\left(A\right)\) (as in Definition~\ref{dfn:invcber}) satisfies the {\bf mass transport principle}:
\[\nu\big(\overrightarrow{F}\big)=\nu\big(\overleftarrow{F}\big)\text{ for every Borel function }F:E\to\left[0,+\infty\right],\]
where the Borel functions \(\overrightarrow{F},\overleftarrow{F}:A\to\left[0,+\infty\right]\) are defined by
\[\overrightarrow{F}\left(a\right)\coloneqq\sum\nolimits_{a^{\prime}\in E\left(a\right)}F\left(a,a^{\prime}\right)\quad\text{and}\quad\overleftarrow{F}\left(a\right)\coloneqq\sum\nolimits_{a^{\prime}\in E\left(a\right)}F\left(a^{\prime},a\right).\]
\end{fct}

The second way we exploit the Lusin--Novikov theorem~\ref{thm:lusinnovi} is in both parts of the following lemma, which together with Fact~\ref{fct:mtp} is essentially all we need to prove Theorem~\ref{thm:relativemecke}.

\begin{lem}\label{lem:lusinnovi}
Let \(X\) be a Borel \(G\)-space with a cross section \(Y\), and consider the associated countable Borel equivalence relation \(E_{G}^{Y}\).
\begin{enumerate}
    \item For every element \(\tau\in\bigl[\!\bigl[E_{G}^{Y}\bigr]\!\bigr]\) there exists a Borel function \(f:G\times Y\to\left[0,+\infty\right]\) solving
    \[f_{Y}\left(y\right)=\mathbf{1}_{\mathrm{Dom}\left(\tau\right)}\left(y\right)\quad\text{and}\quad f_{X}\left(y\right)=\mathbf{1}_{\mathrm{Rng}\left(\tau\right)}\left(y\right)\quad\text{for all }y\in Y.\]
    \item For every Borel function \(f:G\times Y\to\left[0,+\infty\right]\) there exists a Borel function \(F:E_{G}^{Y}\to\left[0,+\infty\right]\) solving
    \[\overrightarrow{F}\left(y\right)=f_{Y}\left(y\right)\quad\text{and}\quad\overleftarrow{F}\left(y\right)=f_{X}\left(y\right)\quad\text{for all }y\in Y.\]
\end{enumerate}
\end{lem}

\begin{proof}[Proof of Lemma~\ref{lem:lusinnovi}]
For (1), given some \(\tau\in\bigl[\!\bigl[E_{G}^{Y}\bigr]\!\bigr]\), consider the Borel set
\[P\coloneqq\left\{ \left(g,y\right)\in G\times\mathrm{Dom}\left(\tau\right):g.y=\tau\left(y\right)\right\} \subseteq G\times Y.\]
Since \(\mathrm{Graph}\left(\tau\right)\subseteq E_{G}^{Y}\), for every \(y\in\mathrm{Dom}\left(\tau\right)\) the corresponding \(y\)-fiber of \(P\) satisfies
\[P_{y}=\left\{ g\in G:g.y=\tau\left(y\right)\right\} \subseteq Y_{y},\]
and hence is countable and nonempty. By the Lusin--Novikov theorem~\ref{thm:lusinnovi},\footnote{Here we only use the first part of~\cite[Theorem (18.10)]{kechris2012classical}.} there is a Borel map \(\lambda:\mathrm{Dom}\left(\tau\right)\to G\) such that
\(\tau\left(y\right)=\lambda\left(y\right).y\) for all \(y\in\mathrm{Dom}\left(\tau\right)\). Define then the Borel function
\[f:G\times Y\longrightarrow\left[0,+\infty\right],\quad f\left(g,y\right)\coloneqq\mathbf{1}_{\mathrm{Dom}\left(\tau\right)}\left(y\right)\cdot\mathbf{1}_{\left\{g=\lambda\left(y\right)\right\}}.\]
Let us show that \(f\) is the desired function. First, for every \(y\in Y\) we have
\[f_{Y}\left(y\right)=\sum\nolimits_{g\in Y_{y}}f\left(g,y\right)=\mathbf{1}_{\mathrm{Dom}\left(\tau\right)}\left(y\right)\cdot\sum\nolimits_{g\in Y_{y}}\mathbf{1}_{\left\{g=\lambda\left(y\right)\right\}}=\mathbf{1}_{\mathrm{Dom}\left(\tau\right)}\left(y\right).\]
Second, for every \(y\in Y\) we claim that
\[f_{X}\left(y\right)=\sum\nolimits_{g\in Y_{y}}f\left(g^{-1},g.y\right)=\sum\nolimits_{g\in Y_{y}}\mathbf{1}_{\mathrm{Dom}\left(\tau\right)}\left(g.y\right)\cdot\mathbf{1}_{\left\{g^{-1}=\lambda\left(g.y\right)\right\}}=\mathbf{1}_{\mathrm{Rng}\left(\tau\right)}\left(y\right).\]
Indeed, if \(y\notin\mathrm{Rng}\left(\tau\right)\) then \(\mathbf{1}_{\mathrm{Dom}\left(\tau\right)}\left(g.y\right)\cdot\mathbf{1}_{\left\{ g^{-1}=\lambda\left(g.y\right)\right\} }=0\) for every \(g\in Y_{y}\), for otherwise we would have \(y=g^{-1}g.y=\lambda\left(g.y\right)g.y=\tau\left(g.y\right)\in\mathrm{Rng}\left(\tau\right)\). Now if \(y\in\mathrm{Rng}\left(\tau\right)\) then \(y=\tau\left(y^{\prime}\right)=\lambda\left(y^{\prime}\right).y^{\prime}\) for some \(y^{\prime}\in\mathrm{Dom}\left(\tau\right)\), so \(g\coloneqq \lambda\left(y^{\prime}\right)^{-1}\) gives \(\mathbf{1}_{\mathrm{Dom}\left(\tau\right)}\left(g.y\right)\cdot\mathbf{1}_{\left\{g^{-1}=\lambda\left(g.y\right)\right\}}=1\); then whenever \(g\in Y_{y}\) satisfies \(\mathbf{1}_{\mathrm{Dom}\left(\tau\right)}\left(g.y\right)\cdot\mathbf{1}_{\left\{g^{-1}=\lambda\left(g.y\right)\right\}}=1\) we have \(\tau\left(g.y\right)=\lambda\left(g.y\right)g.y=g^{-1}g.y=y\), so two such \(g,g^{\prime}\) satisfy \(\tau\left(g.y\right)=\tau\left(g^{\prime}.y\right)\), hence \(g.y=g^{\prime}.y\) (as \(\tau\) is injective). Then \(g^{-1}=\lambda\left(g.y\right)=\lambda\left(g^{\prime}.y\right)=g^{\prime-1}\), hence \(g=g^{\prime}\).

For (2), consider the Borel set
\[Q\coloneqq\left\{ \left(g,y,y^{\prime}\right)\in G\times E_{G}^{Y}:g.y=y^{\prime}\right\} \subseteq G\times E_{G}^{Y}.\]
For every \(\left(y,y^{\prime}\right)\in E_{G}^{Y}\) the corresponding \(\left(y,y^{\prime}\right)\)-fiber of \(Q\) satisfies
\[Q_{\left(y,y^{\prime}\right)}=\left\{ g\in G:g.y=y^{\prime}\right\} \subseteq Y_{y},\]
and hence is countable, and it is nonempty (since already \(\left(y,y^{\prime}\right)\in E_{G}^{Y}\)). It follows from the Lusin--Novikov theorem~\ref{thm:lusinnovi} that there exist Borel maps \(\lambda_{n}:E_{G}^{Y}\to G\), \(n\in\mathbb{N}\), such that \(Q_{\left(y,y^{\prime}\right)}=\left\{ \lambda_{n}\left(y,y^{\prime}\right):n\in\mathbb{N}\right\}\) for every \(\left(y,y^{\prime}\right)\in E_{G}^{Y}\). Define Borel indicators \(\kappa_{n}:E_{G}^{Y}\to\left\{ 0,1\right\}\), \(n\in\mathbb{N}\), by letting \(\kappa_{1}\equiv 1\) and inductively
\[\kappa_{n+1}\left(y,y^{\prime}\right)\coloneqq\mathbf{1}_{\left\{ \lambda_{n+1}\left(y,y^{\prime}\right)\notin\left\{ \lambda_{1}\left(y,y^{\prime}\right),\dotsc,\lambda_{n}\left(y,y^{\prime}\right)\right\} \right\} },\quad\left(y,y^{\prime}\right)\in E_{G}^{Y}.\]
Now given some Borel function \(f:G\times Y\to\left[0,+\infty\right]\), define the Borel function
\[F:E_{G}^{Y}\longrightarrow\left[0,+\infty\right],\quad F\left(y,y^{\prime}\right)\coloneqq\sum\nolimits_{n\in\mathbb{N}}\kappa_{n}\left(y,y^{\prime}\right)\cdot f\left(\lambda_{n}\left(y,y^{\prime}\right),y\right).\]
Let us show that \(F\) is the desired function. First, by the construction of the \(\kappa_{n}\)'s, for every \(y\in Y\) we have a representation of \(Y_{y}\) as the disjoint union
\[Y_{y}=\bigsqcup\nolimits_{y^{\prime}\in E_{G}^{Y}\left(y\right)}\bigsqcup\nolimits_{\left\{ n\in\mathbb{N}:\kappa_{n}\left(y,y^{\prime}\right)=1\right\} }\left\{ \lambda_{n}\left(y,y^{\prime}\right)\right\},\]
and therefore
\[\overrightarrow{F}\left(y\right)=\sum\nolimits_{y^{\prime}\in E_{G}^{Y}\left(y\right)}\sum\nolimits_{n\in\mathbb{N}}\kappa_{n}\left(y,y^{\prime}\right)\cdot f\left(\lambda_{n}\left(y,y^{\prime}\right),y\right)=\sum\nolimits_{g\in Y_{y}}f\left(g,y\right)=f_{Y}\left(y\right).\]
Second, one can routinely verify that for every \(y\in Y\) we have a bijection
\[Y_{y}\to\bigsqcup\nolimits_{y^{\prime}\in E_{G}^{Y}\left(y\right)}\bigsqcup\nolimits_{\left\{ n\in\mathbb{N}:\kappa_{n}\left(y,y^{\prime}\right)=1\right\} }\left\{ \left(\lambda_{n}\left(y^{\prime},y\right),y^{\prime}\right)\right\}\quad\text{given by}\quad g\mapsto\left(g^{-1},g.y\right),\]
and therefore
\[\overleftarrow{F}\left(y\right)=\sum\nolimits_{n\in\mathbb{N}}\sum\nolimits_{y^{\prime}\in E_{G}^{Y}\left(y\right)}\kappa_{n}\left(y^{\prime},y\right)\cdot f\left(\lambda_{n}\left(y^{\prime},y\right),y^{\prime}\right)=\sum\nolimits_{g\in Y_{y}}f\left(g^{-1},g.y\right)=f_{X}\left(y\right).\qedhere\]
\end{proof}

\begin{proof}[Proof of Theorem~\ref{thm:relativemecke}]
For the inclusion \(\mathcal{M}_{\mathrm{Mecke}}^{G}\left(Y\right)\subseteq\mathcal{M}^{E_{G}^{Y}}\left(Y\right)\), let \(\nu\in\mathcal{M}_{\mathrm{Mecke}}^{G}\left(Y\right)\), and for an arbitrary element \(\tau\in\bigl[\!\bigl[E_{G}^{Y}\bigr]\!\bigr]\), use Lemma~\ref{lem:lusinnovi}(1) to pick a Borel function \(f:G\times Y\to\left[0,+\infty\right]\) solving \(f_{Y}=\mathbf{1}_{\mathrm{Dom}\left(\tau\right)}\) and \(f_{X}=\mathbf{1}_{\mathrm{Rng}\left(\tau\right)}\). Since \(\nu\) satisfies the Mecke equation,
\[\nu\left(\mathrm{Dom}\left(\tau\right)\right)=\nu\left(f_{Y}\right)=\nu\left(f_{X}\right)=\nu\left(\mathrm{Rng}\left(\tau\right)\right).\]
As \(\tau\) is arbitrary, we conclude that \(\nu\in\mathcal{M}^{E_{G}^{Y}}\left(Y\right)\).

For the inclusion \(\mathcal{M}_{\mathrm{Mecke}}^{G}\left(Y\right)\supseteq\mathcal{M}^{E_{G}^{Y}}\left(Y\right)\) we use the mass transport principle as in Fact~\ref{fct:mtp}. Let \(\nu\in\mathcal{M}^{E_{G}^{Y}}\left(Y\right)\). For an arbitrary Borel function \(f:G\times Y\to\left[0,+\infty\right]\), use Lemma~\ref{lem:lusinnovi}(2) to pick a Borel function \(F:E_{G}^{Y}\to\left[0,+\infty\right]\) solving \(\overrightarrow{F}=f_{Y}\) and \(\overleftarrow{F}=f_{X}\). Since \(\nu\) satisfies the mass transport principle,
\[\nu\left(f_{Y}\right)=\nu\big(\overrightarrow{F}\big)=\nu\big(\overleftarrow{F}\big)=\nu\left(f_{X}\right),\]
which is the Mecke equation. As \(f\) is arbitrary, we conclude that \(\nu\in\mathcal{M}_{\mathrm{Mecke}}^{G}\left(Y\right)\).
\end{proof}

\section{The transverse correspondence}\label{sct:transverse}

In the following discussion we benefit significantly from Last's account of the classical Palm measure theory~\cite{last2009} (see also~\cite{connes1979},~\cite[Appendix A.1]{anantharaman2000},~\cite[\S2]{avni2010entropy},~\cite[Appendix B]{kyed2015},~\cite[\S4]{Slutsky2017},~\cite[\S1.5]{bjorklund2025int}, which deal with \emph{separated} cross sections). We start with the general definition of the transverse measure.

\begin{defn}\label{dfn:transdef}
Let \(X\) be a Borel \(G\)-space with a cross section \(Y\). For every \(\mu\in\mathcal{M}^{G}\left(X\right)\), the associated {\bf transverse measure} \(\mu_{Y}\in\mathcal{M}\left(Y\right)\) is defined by
\[\mu_{Y}\left(f\right)\coloneqq\int_{X}\sum\nolimits_{g\in Y_{x}}f\left(g.x\right)w\left(g\right)d\mu\left(x\right),\]
for every Borel function \(f:Y\to\left[0,+\infty\right]\), for some Borel function \(w:G\to\left[0,+\infty\right]\) with \(m_{G}\left(w\right)=1\).
\end{defn}

It turns out that \(\mu_{Y}\) is independent of the choice of \(w\). This is seen through the fundamental property called in point processes \emph{Campbell theorem} (or \emph{refined Campbell theorem}; see~\cite{last2009}), and we keep that name. In order to formulate it, recall the \(X\){\bf-periodization} of a Borel function \(f:G\times Y\to\left[0,+\infty\right]\), which is the Borel function
\[f_{X}:X\longrightarrow\left[0,+\infty\right],\quad f_{X}\left(x\right)\coloneqq \sum\nolimits_{g\in Y_{x}}f\left(g^{-1},g.x\right).\]

\begin{prop}[{\bf Campbell theorem}]\label{prop:campbell}
For every \(\mu\in\mathcal{M}^{G}\left(X\right)\), the transverse measure \(\mu_{Y}\) is \(\sigma\)-finite and satisfies
\[m_{G}\otimes\mu_{Y}\left(f\right)=\mu\left(f_{X}\right)\]
for every Borel function \(f:G\times Y\to\left[0,+\infty\right]\). In particular, \(\mu_{Y}\) is independent of the choice of \(w\).
\end{prop}

We now formulate the transverse correspondence, and to this end we recall the set \(\mathcal{M}^{E_{G}^{Y}}\left(Y\right)\) of all \(E_{G}^{Y}\)-invariant Borel \(\sigma\)-finite measures on \(Y\).

\begin{thm}[The transverse correspondence]\label{thm:corresp}
For every Borel \(G\)-space \(X\) with a cross section \(Y\), the transverse measure construction as in Definition~\ref{dfn:transdef} constitutes a canonical bijection
\[\mathcal{M}^{G}\left(X\right)\longrightarrow\mathcal{M}^{E_{G}^{Y}}\left(Y\right),\quad\mu\mapsto\mu_{Y}.\]
\end{thm}

Note that the transverse correspondence depends on a specific normalization of \(G\)'s Haar measure.

\subsection{Proof of Campbell theorem}

In the following proof we follow the lines of~\cite[Theorem 3.6]{last2009}.

\begin{proof}[Proof of Proposition~\ref{prop:campbell}]
Let \(f:G\times Y\to\left[0,+\infty\right]\) be an arbitrary Borel function, and define
\[J_{f}\coloneqq\int_{Y}\Big(\int_{G}f\left(g,y\right)dm_{G}\left(g\right)\Big)d\mu_{Y}\left(y\right).\]
By the definition of \(\mu_{Y}\) and Fubini's theorem for \(m_{G}\otimes\mu\) (but not with \(\mu_{Y}\), which has not yet been proved to be \(\sigma\)-finite), we get that
\begin{align*}
J_{f}
&=\iint\nolimits_{G\times X}\sum\nolimits_{h\in Y_{x}}f\left(g,h.x\right)w\left(h\right)dm_{G}\left(g\right)d\mu\left(x\right)\\
&=\iint\nolimits_{G\times X}\sum\nolimits_{h\in Y_{x}}f\left(g^{-1}h^{-1},h.x\right)w\left(h\right)dm_{G}\left(g\right)d\mu\left(x\right)\\
&=\iint\nolimits_{G\times X}\sum\nolimits_{h\in Y_{g^{-1}.x}}f\left(h^{-1},hg^{-1}.x\right)w\left(hg^{-1}\right)dm_{G}\left(g\right)d\mu\left(x\right),
\end{align*}
where in the last equality we substitute \(h\mapsto hg^{-1}\). Then using Fubini's theorem twice, the \(G\)-invariance of \(\mu\), and that \(\int_{G}w\left(hg^{-1}\right)dm_{G}\left(g\right)=1\) for every \(h\in G\) by unimodularity, we obtain the identity
\begin{equation}\label{eq:mididentity}
\begin{aligned}
J_{f}
&=\iint\nolimits_{G\times X}\sum\nolimits_{h\in Y_{x}}f\left(h^{-1},h.x\right)w\left(hg^{-1}\right)dm_{G}\left(g\right)d\mu\left(x\right)\\
&=\int\nolimits_{X}\sum\nolimits_{h\in Y_{x}}f\left(h^{-1},hx\right)d\mu\left(x\right)=\mu\left(f_{X}\right).
\end{aligned}
\end{equation}
In order to complete the proof we need to justify that \(J_{f}=m_{G}\otimes\mu_{Y}\left(f\right)\), which will follow directly from the definition of \(J_{f}\) by Fubini's theorem once we prove that \(\mu_{Y}\) is \(\sigma\)-finite. Therefore, for the rest of the proof we exploit the identity \eqref{eq:mididentity} to show that \(\mu_{Y}\) is \(\sigma\)-finite.

Pick some relatively compact symmetric identity neighborhood \(K\subset G\). We first note that necessarily \(\mu\left(x\in X:\left|Y_{x}\cap K\right|\geq1\right)>0\); indeed, otherwise, since \(\mu\) is \(G\)-invariant, then for every \(g\in G\) we have
\[0=\mu\left(x\in X:\left|Y_{x}\cap K\right|\geq1\right)=\mu\left(x\in X:\left|Y_{g.x}\cap K\right|\geq1\right)=\mu\left(x\in X:\left|Y_{x}\cap Kg\right|\geq1\right),\]
while \(G\) can be covered using countably many right-translations of \(K\) (by the Lindel\"{o}f property), so this is impossible. We also note that \(\left|Y_{x}\cap K\right|<+\infty\) for every \(x\in X\), since \(Y_{x}\) is locally finite, and therefore, using the \(\sigma\)-finiteness of \(\mu\), there is a Borel function \(\theta:X\to\left(0,+\infty\right)\) such that
\[0<\int_{X}\theta\left(x\right)\cdot\left|Y_{x}\cap K\right|d\mu\left(x\right)<+\infty.\]
Define the Borel function
\[f:G\times Y\to\left[0,+\infty\right),\quad f\left(g,y\right)\coloneqq\theta\left(g.y\right)\cdot\mathbf{1}_{K}\left(g\right),\]
and note that its \(X\)-periodization is
\[f_{X}\left(x\right)=\sum\nolimits_{g\in Y_{x}}f\left(g^{-1},g.x\right)=\sum\nolimits_{g\in Y_{x}}\theta\left(x\right)\cdot\mathbf{1}_{K}\left(g^{-1}\right)=\theta\left(x\right)\cdot\left|Y_{x}\cap K^{-1}\right|=\theta\left(x\right)\cdot\left|Y_{x}\cap K\right|.\]
It follows from the identity \eqref{eq:mididentity} and the choice of \(K\) and \(\theta\) that
\[\int_{Y}\Big(\int_{K}\theta\left(g.y\right)dm_{G}\left(g\right)\Big)d\mu_{Y}\left(y\right)=J_{f}=\mu\left(f_{X}\right)=\int_{X}\theta\left(x\right)\cdot\left|Y_{x}\cap K\right|d\mu\left(x\right)<+\infty.\]
Therefore, the Borel function \(Y\to\left(0,+\infty\right)\), \(y\mapsto\int_{K}\theta\left(g.y\right)dm_{G}\left(g\right)\), is \(\mu_{Y}\)-integrable while taking values in \(\left(0,+\infty\right)\) (since \(\theta\) is strictly positive and \(m_{G}\left(K\right)>0\)), and thus \(\mu_{Y}\) is \(\sigma\)-finite.
\end{proof}

We now deduce that the range of the transverse correspondence~\ref{thm:corresp} is indeed \(\mathcal{M}^{E_{G}^{Y}}\left(Y\right)\). To this end we argue similarly to~\cite[Theorem 3.15]{last2009} and exploit Theorem~\ref{thm:relativemecke}.

\begin{prop}\label{prop:mecke}
For every \(\mu\in\mathcal{M}^{G}\left(X\right)\), we have \(\mu_{Y}\in \mathcal{M}^{E_{G}^{Y}}\left(Y\right)\).
\end{prop}

\begin{proof}
By Theorem~\ref{thm:relativemecke} it suffices to show that \(\mu_{Y}\in \mathcal{M}_{\mathrm{Mecke}}^{G}\left(Y\right)\). Let \(f:G\times Y\to\left[0,+\infty\right]\) be a Borel function. Since \(\int_{G}w\left(hg^{-1}\right)dm_{G}\left(h\right)=1\) for every \(g\in G\) by unimodularity, we can write
\[\mu_{Y}\left(f_{X}\right)=\iint\nolimits_{G\times Y}\sum\nolimits_{g\in Y_{y}}f\left(g^{-1},g.y\right)w\left(gh^{-1}\right)dm_{G}\left(h\right)d\mu_{Y}\left(y\right)\coloneqq m_{G}\otimes\mu_{Y}\left(F\right),\]
where \(F\left(h,y\right)\coloneqq \sum\nolimits_{g\in Y_{y}}f\left(g^{-1},g.y\right)w\left(gh^{-1}\right)\). For every \(x\in X\) and \(h\in Y_{x}\) we have
\[F\left(h^{-1},h.x\right)=\sum\nolimits_{gh\in Y_{x}}f\left(g^{-1},gh.x\right)w\left(gh\right)=\sum\nolimits_{g\in Y_{x}}f\left(hg^{-1},g.x\right)w\left(g\right),\]
so that the \(X\)-periodization of \(F\) becomes
\[F_{X}\left(x\right)=\sum\nolimits_{g\in Y_{x}}\Big(\sum\nolimits_{h\in Y_{x}}f\left(hg^{-1},g.x\right)\Big)w\left(g\right)=\sum\nolimits_{g\in Y_{x}}f_{Y}\left(g.x\right)w\left(g\right),\]
where in the second equality we substitute \(h\mapsto hg\) in the inner sum. Then by the definition of \(\mu_{Y}\) we get \(\mu\left(F_{X}\right)=\mu_{Y}\left(f_{Y}\right)\), and using Campbell theorem~\ref{prop:campbell} we obtain the Mecke equation:
\[\mu_{Y}\left(f_{X}\right)=m_{G}\otimes\mu_{Y}\left(F\right)=\mu\left(F_{X}\right)=\mu_{Y}\left(f_{Y}\right).\qedhere\]
\end{proof}

\subsection{Injective covers}

Here we construct \emph{injective covers} for Borel \(G\)-spaces with cross sections. This will be important to show that the transverse correspondence~\ref{thm:corresp} is bijective. For separated cross sections, injective covers were constructed in~\cite[\S3.1]{bjorklund2025int}, and here we extend it to general cross sections.

\begin{defn}
Let \(X\) be a Borel \(G\)-space with a cross section \(Y\) and action map \(\mathrm{a}:G\times X\to X\).
\begin{itemize}
    \item An {\bf injective set} is a Borel set \(C\subseteq G\times Y\) such that \(\mathrm{a}\mid_{C}:C\to\mathrm{a}\left(C\right)\) is an injective map.
    \item An {\bf injective cover} is a countable family \(\mathfrak{C}\) of Borel sets in \(G\times Y\) such that:
    \begin{enumerate}
    \item Each element of \(\mathfrak{C}\) is an injective set.
    \item The elements of \(\mathfrak{C}\) are pairwise disjoint.
    \item \(\left\{\mathrm{a}\left(C\right):C\in\mathfrak{C}\right\}\) forms a (countable) Borel partition of \(X\).
    \end{enumerate}
\end{itemize}
\end{defn}

\begin{lem}\label{lem:injcov}
Every Borel \(G\)-space \(X\) with a cross section \(Y\) admits an injective cover.
\end{lem}

\begin{proof}
Fix an enumerated countable base \(\mathcal{U}=\left\{ U_{n}:n\in\mathbb{N}\right\}\) for the topology of \(G\), and set
\[X_{1}\left(U_{n}\right)\coloneqq \{x\in X: \left|Y_{x}\cap U_{n}\right|=1\}.\]
Note that for every \(x\in X\) the set \(Y_{x}\) is locally finite in \(G\), and thus \(X=\bigcup_{n\in\mathbb{N}}X_{1}\left(U_{n}\right)\). Then form the Borel partition
\[X=X_{1}\sqcup X_{2}\sqcup\dotsm,\text{ where }X_{n+1}\coloneqq X_{1}\left(U_{n+1}\right)\setminus\left(X_{1}\left(U_{n}\right)\cup\dotsm\cup X_{1}\left(U_{1}\right)\right).\]
Now define
\[C_{n}\coloneqq \mathrm{a}^{-1}\left(X_{n}\right)\cap\left(U_{n}^{-1}\times Y\right)=\left\{\left(g,y\right)\in G\times Y: g^{-1}\in U_{n},\,g.y\in X_{n}\right\},\qquad n\in\mathbb{N},\]
and we claim that \(C_{n}\) is an injective set satisfying \(\mathrm{a}\left(C_{n}\right)=X_{n}\). To see this, suppose \(\left(g,y\right),\left(g^{\prime},y^{\prime}\right)\in C_{n}\) are such that \(x\coloneqq g.y=g^{\prime}.y^{\prime}\in X_{n}\). Then \(g^{-1},g^{\prime -1}\in Y_{x}\cap U_{n}\), and since \(X_{n}\subseteq X_{1}\left(U_{n}\right)\) this means that \(g^{-1}=g^{\prime -1}\), hence \(g=g^{\prime}\) and \(y=y^{\prime}\), concluding that \(C_{n}\) is an injective set. It is clear that \(\mathrm{a}\left(C_{n}\right)\subseteq X_{n}\), and conversely, if \(x\in X_{n}\) and \(g\in Y_{x}\cap U_{n}\) is the unique element of this set, then \(\left(g^{-1},g.x\right)\in C_{n}\) hence \(x=g^{-1}.g.x\in\mathrm{a}\left(C_{n}\right)\), so that \(X_{n}\subseteq\mathrm{a}\left(C_{n}\right)\). Since \(X_{1},X_{2},X_{3},\dotsc\) are pairwise disjoint so are \(C_{1},C_{2},C_{3},\dotsc\), and thus \(\mathfrak{C}\coloneqq \{C_{1},C_{2},\dotsc\}\) forms an injective cover.
\end{proof}

\begin{lem}\label{lem:borpart}
Every Borel \(G\)-space \(X\) with a cross section \(Y\) admits a {\bf Borel partition of unity}, that is a Borel function
\[\rho:G\times X\longrightarrow\left[0,1\right]\,\text{ satisfying }\,\sum\nolimits_{g\in Y_{x}}\rho\left(g,x\right)=1\text{ for all }x\in X.\]
\end{lem}

\begin{proof}
Fix an injective cover \(\mathfrak{C}\) using Lemma~\ref{lem:injcov}, and define
\[\rho\left(g,x\right)\coloneqq\sum\nolimits_{C\in\mathfrak{C}}\mathbf{1}_{C}\left(g^{-1},g.x\right).\]
By the defining properties of injective covers, for every \(x\in X\) there is a unique \(C\in\mathfrak{C}\) and a unique \(g\in Y_{x}\) such that \(\mathbf{1}_{C}\left(g^{-1},g.x\right)=1\). This readily implies that \(\sum\nolimits_{g\in Y_{x}}\rho\left(g,x\right)=1\).
\end{proof}

The following lemma refines Lemma~\ref{lem:borpart}.

\begin{lem}\label{lem:borpartmon}
Let \(X^{\prime}\) be a Borel \(G\)-space with a cross section \(Y^{\prime}\), and let \(X\subseteq X^{\prime}\) be a Borel \(G\)-subspace (i.e., a \(G\)-invariant Borel subset) with a cross section \(Y\), such that \(Y\subseteq Y^{\prime}\). Then there are Borel partitions of unity, \(\rho:G\times X\to\left[0,1\right]\) for \(X\) with respect to \(Y\) and \(\rho^{\prime}:G\times X^{\prime}\to\left[0,1\right]\) for \(X^{\prime}\) with respect to \(Y^{\prime}\), such that
\[\rho\left(g,x\right)=\rho^{\prime}\left(g,x\right)\text{ for all }\left(g,x\right)\in G\times X.\]
\end{lem}

\begin{proof}
Since \(X\) is \(G\)-invariant and \(Y^{\prime}\) is a cross section for \(X^{\prime}\), the set \(Y^{\prime}\setminus X\) is a cross section for \(X^{\prime}\setminus X\). By Lemma~\ref{lem:injcov}, there is an injective cover \(\mathfrak{C}\) of \(X\) with respect to \(Y\), and an injective cover \(\mathfrak{D}\) of \(X^{\prime}\setminus X\) with respect to \(Y^{\prime}\setminus X\). Since \(Y\subseteq X\) and \(Y^{\prime}\setminus X\subseteq X^{\prime}\setminus X\) are disjoint, we have \(C\cap D=\varnothing\) for all \(C\in\mathfrak{C}\) and \(D\in\mathfrak{D}\). Hence \(\mathfrak{C}^{\prime}\coloneqq \mathfrak{C}\cup\mathfrak{D}\) is an injective cover of \(X^{\prime}\) with respect to \(Y^{\prime}\).

Consider the Borel partitions of unity \(\rho:G\times X\to\left[0,1\right]\) associated with \(\mathfrak{C}\) and \(\rho^{\prime}:G\times X^{\prime}\to\left[0,1\right]\) associated with \(\mathfrak{C}^{\prime}\), defined as in Lemma~\ref{lem:borpart} by
\[\rho\left(g,x\right)\coloneqq\sum\nolimits_{C\in\mathfrak{C}}\mathbf{1}_{C}\!\left(g^{-1},g.x\right)\quad\text{and}\quad
\rho^{\prime}\left(g,x^{\prime}\right)\coloneqq\sum\nolimits_{C^{\prime}\in\mathfrak{C}^{\prime}}\mathbf{1}_{C^{\prime}}\left(g^{-1},g.x^{\prime}\right).\]
Fix \(\left(g,x\right)\in G\times X\). For every \(D\in\mathfrak{D}\) we have \(D\subseteq G\times\left(Y^{\prime}\setminus X\right)\), while \(g.x\in X\) (by \(G\)-invariance of \(X\)), hence \(\mathbf{1}_{D}\left(g^{-1},g.x\right)=0\). Therefore,
\[\rho^{\prime}\left(g,x\right)
=\sum\nolimits_{C\in\mathfrak{C}}\mathbf{1}_{C}\left(g^{-1},g.x\right)
+\sum\nolimits_{D\in\mathfrak{D}}\mathbf{1}_{D}\left(g^{-1},g.x\right)
=\rho\left(g,x\right). \qedhere\]
\end{proof}

Using injective sets, Campbell theorem~\ref{prop:campbell} can be reformulated as follows:

\begin{prop}[Campbell theorem, revisited]
\label{prop:campbellrevi}
For every injective set \(C\subset G\times Y\), the action map
\[\mathrm{a}\mid_{C}:\left(C,\left(m_{G}\otimes\mu_{Y}\right)\mid_{C}\right)\to\left(\mathrm{a}\left(C\right),\mu\mid_{\mathrm{a}\left(C\right)}\right)\]
is measure preserving, meaning that it pushes \(\left(m_{G}\otimes\mu_{Y}\right)\mid_{C}\) to \(\mu\mid_{\mathrm{a}\left(C\right)}\).
\end{prop}

\begin{proof}
For a Borel set \(C\subset G\times Y\), being an injective set means that the \(X\)-periodization of \(f\coloneqq \mathbf{1}_{C}\) is
\[f_{X}\left(x\right)=\sum\nolimits_{g\in Y_{x}}\mathbf{1}_{C}\left(g^{-1},g.x\right)=\mathbf{1}_{\mathrm{a}\left(C\right)}\left(x\right),\quad x\in X.\]
Then by virtue of Campbell theorem~\ref{prop:campbell}, \(m_{G}\otimes\mu_{Y}\left(C\right)=\mu\left(\mathrm{a}\left(C\right)\right)\). Since a subset of an injective set is again injective, this means that \(m_{G}\otimes\mu_{Y}\left(E\right)=\mu\left(\mathrm{a}\left(E\right)\right)\) for every Borel set \(E\subseteq C\).
\end{proof}

Using injective covers we can establish the injectivity of the transverse correspondence~\ref{thm:corresp}.

\begin{prop}\label{prop:injcorres}
The transverse correspondence \(\mathcal{M}^{G}\left(X\right)\to\mathcal{M}^{E_{G}^{Y}}\left(Y\right)\), \(\mu\mapsto\mu_{Y}\), is injective. Moreover, for every \(\nu\in\mathcal{M}\left(Y\right)\) there is at most one \(\mu\in\mathcal{M}\left(X\right)\) satisfying Campbell theorem~\ref{prop:campbell} with \(\nu\):
\[m_{G}\otimes\nu\left(f\right)=\mu\left(f_{X}\right)\]
for every Borel function \(f:G\times Y\to\left[0,+\infty\right]\).
\end{prop}

\begin{proof}
Let \(\mu\in\mathcal{M}\left(X\right)\) be arbitrary. Suppose there is \(\nu\in\mathcal{M}\left(Y\right)\) such that \(\mu\) satisfies Campbell theorem with \(\nu\). Fix an injective cover \(\mathfrak{C}\) using Lemma~\ref{lem:injcov}. Using the same reasoning as in Proposition~\ref{prop:campbellrevi}, for every Borel set \(A\subseteq X\) we have
\[\mu\left(A\cap\mathrm{a}\left(C\right)\right)=m_{G}\otimes\nu\left(\mathrm{a}\mid_{C}^{-1}\left(A\right)\cap C\right),\quad C\in\mathfrak{C},\]
and therefore, summing over \(C\in\mathfrak{C}\),
\[\mu\left(A\right)=\sum\nolimits_{C\in\mathfrak{C}}m_{G}\otimes\nu\left(\mathrm{a}\mid_{C}^{-1}\left(A\right)\cap C\right).\]
Thus, \(\mu\) is completely recovered from \(\nu\).
\end{proof}

\subsection{Final proof of the transverse correspondence}

Let us start by describing the inverse correspondence of the transverse correspondence~\ref{thm:corresp}:
\begin{equation}\label{eq:invcorr}
\mathcal{M}^{E_{G}^{Y}}\left(Y\right)=\mathcal{M}_{\mathrm{Mecke}}^{G}\left(Y\right)\to\mathcal{M}^{G}\left(X\right),\quad \nu\mapsto\nu^{X}.
\end{equation}
Given \(\nu\in\mathcal{M}_{\mathrm{Mecke}}^{G}\left(Y\right)\), pick a Borel partition of unity \(\rho\) using Lemma~\ref{lem:borpart}, and define \(\nu^{X}\in\mathcal{M}\left(X\right)\) by
\[\nu^{X}\left(f\right)\coloneqq \iint\nolimits_{G\times Y}\rho\left(g,g^{-1}.y\right)f\left(g^{-1}.y\right)dm_{G}\left(g\right)d\nu\left(y\right)\]
for every Borel function \(f:X\to\left[0,+\infty\right]\). Let us verify the basic properties of those \(\nu\) and \(\nu^{X}\).

\begin{claim}\label{clm:campbcorr}
The measures \(\nu\) and \(\nu^{X}\) satisfies Campbell theorem together:
\[\nu^{X}\left(f_{X}\right)=m_{G}\otimes\nu\left(f\right)\]
for every Borel function \(f:G\times Y\to\left[0,+\infty\right]\). Therefore, \(\nu^{X}\) is \(\sigma\)-finite and its definition is independent on the choice of \(\rho\).
\end{claim}

\begin{proof}
Let \(f:G\times Y\to\left[0,+\infty\right]\) be a Borel function. For every \(g\in G\) define
\[f^{g}\left(h,y\right)\coloneqq \rho\left(g,g^{-1}.y\right)f\left(g^{-1}h^{-1},h.y\right),\quad h\in Y_{y},\]
with the fixed Borel partition of unity \(\rho\) defining \(\nu^{X}\), and then
\[f^{g}\left(h^{-1},h.y\right)=\rho\left(g,g^{-1}h.y\right)f\left(g^{-1}h,y\right),\quad h\in Y_{y}.\]
We now claim that the desired identity follows via
\begin{align*}
\nu^{X}\left(f_{X}\right)
&=\iint\nolimits_{G\times Y}\rho\left(g,g^{-1}.y\right)f_{X}\left(g^{-1}.y\right)dm_{G}\left(g\right)d\nu\left(y\right)\\
\left(\mathrm{i}\right)&=\iint\nolimits_{G\times Y}\sum\nolimits_{h\in Y_{y}}f^{g}\left(h,y\right)dm_{G}\left(g\right)d\nu\left(y\right)\\
\left(\mathrm{ii}\right)&=\iint\nolimits_{G\times Y}\sum\nolimits_{h\in Y_{y}}f^{g}\left(h^{-1},h.y\right)dm_{G}\left(g\right)d\nu\left(y\right)\\
&=\iint\nolimits_{G\times Y}\sum\nolimits_{h\in Y_{y}}\rho\left(g,g^{-1}h.y\right)f\left(g^{-1}h,y\right)dm_{G}\left(k\right)d\nu\left(y\right)\\
\left(\mathrm{iii}\right)&=\iint\nolimits_{G\times Y}\sum\nolimits_{g\in Y_{h^{-1}.y}}\rho\left(g,h^{-1}.y\right)f\left(h^{-1},y\right)dm_{G}\left(h\right)d\nu\left(y\right)\\
\left(\mathrm{iv}\right)&=\iint\nolimits_{G\times Y}f\left(h^{-1},y\right)dm_{G}\left(h\right)d\nu\left(y\right)=m_{G}\otimes\nu\left(f\right);
\end{align*}
indeed, for \(\left(\mathrm{i}\right)\) one can directly verify that for every \(g\in G\) and \(y\in Y\),
\[f_{X}\left(g^{-1}.y\right)=\sum\nolimits_{h\in Y_{y}}f\left(g^{-1}h^{-1},h.y\right);\]
for \(\left(\mathrm{ii}\right)\) we applied the Mecke equation to \(f^{g}\) for each \(g\in G\); for \(\left(\mathrm{iii}\right)\) we substitute \(h\mapsto h^{-1}g\) and use unimodularity; finally, for \(\left(\mathrm{iv}\right)\) we used that \(\rho\) is a Borel partition of the unity.

We now deduce that \(\nu^{X}\) is \(\sigma\)-finite. Using that \(m_{G}\otimes\nu\) is \(\sigma\)-finite, pick a Borel function \(f:G\times Y\to\mathbb{R}_{>0}\) (strictly positive) such that \(\nu^{X}\left(f_{X}\right)=m_{G}\otimes\nu\left(f\right)<+\infty\). However, as \(f\) is strictly positive, \(f_{X}\) is strictly positive as well, and it follows that \(\nu^{X}\) is \(\sigma\)-finite. Finally, since \(\nu^{X}\) satisfies Campbell theorem with \(\nu\) for whatever choice of \(\rho\), from Proposition~\ref{prop:injcorres} it follows that \(\nu^{X}\) is independent of the choice of \(\rho\).
\end{proof}

\begin{claim}
\(\nu^{X}\in\mathcal{M}^{G}\left(X\right)\).
\end{claim}

\begin{proof}
Let \(h\in G\) be arbitrary. Fix a Borel partition of unity \(\rho\), and define
\[\rho_{h}:G\times X\longrightarrow\left[0,1\right],\quad\rho_{h}\left(g,x\right)=\rho\left(gh^{-1},h.x\right).\]
Observe that \(\rho_{h}\) is again a Borel partition of the unity. Therefore, by the independence of \(\nu^{X}\) on the choice of the Borel partition of unity as in Claim~\ref{clm:campbcorr}, we have \(\nu^{X}=\nu_{h}^{X}\), where \(\nu_{h}^{X}\) is defined by \(\rho_{h}\) (in the same way \(\nu^{X}\) is defined by \(\rho\)). Then for every Borel function \(f:X\to\left[0,+\infty\right]\), using Fubini's theorem and unimodularity, we have
\begin{align*}
\nu^{X}\left(f\left(h.\cdot\right)\right)
&=\nu_{h}^{X}\left(f\left(h.\cdot\right)\right)\\
&=\iint\nolimits_{G\times Y}\rho_{h}\left(g,g^{-1}.y\right)f\left(hg^{-1}.y\right)dm_{G}\left(g\right)d\nu\left(y\right)\\
&=\iint\nolimits_{G\times Y}\rho\left(gh^{-1},hg^{-1}.y\right)f\left(hg^{-1}.y\right)dm_{G}\left(g\right)d\nu\left(y\right)\\
&=\iint\nolimits_{G\times Y}\rho\left(g^{-1},g.y\right)f\left(g.y\right)dm_{G}\left(g\right)d\nu\left(y\right)=\nu^{X}\left(f\right).\qedhere
\end{align*}
\end{proof}

\begin{proof}[Final proof of Theorem~\ref{thm:corresp}]
In light of Proposition~\ref{prop:mecke} and Theorem~\ref{thm:relativemecke}, we need to show that the correspondence \(\mathcal{M}^{G}\left(X\right)\to\mathcal{M}_{\mathrm{Mecke}}^{G}\left(Y\right)\), \(\mu\mapsto\mu_{Y}\), forms a bijection. The injectivity has been established in Proposition~\ref{prop:injcorres}. For the surjectivity, for every \(\nu\in\mathcal{M}_{\mathrm{Mecke}}^{G}\left(Y\right)\) we have constructed \(\nu^{X}\in\mathcal{M}^{G}\left(X\right)\) via the correspondence \eqref{eq:invcorr}. By Claim~\ref{clm:campbcorr}, \(\nu^{X}\left(f_{X}\right)=m_{G}\otimes\nu\left(f\right)\) for every Borel function \(f:G\times Y\to\left[0,+\infty\right]\), and by Proposition~\ref{prop:injcorres}, this property determines \(\nu_{Y}^{X}\) uniquely among the elements of \(\mathcal{M}_{\mathrm{Mecke}}^{G}\left(Y\right)\), which include \(\nu\), concluding that \(\nu_{Y}^{X}=\nu\).
\end{proof}

\section{Transverse \(G\)-spaces and \(G\)-factors}

Let \(G\) be an lcsc unimodular group. A {\bf transverse \(G\)-space} is a triplet
\[\left(X,\mu,Y\right),\]
where \(\left(X,\mu\right)\) is a probability preserving \(G\)-space and \(Y\) is a {\bf locally integrable cross section}, that is:
\begin{enumerate}
    \item The set \(G.Y\subseteq X\) is a \(\mu\)-conull set, and \(Y_{x}\subset G\) is a locally finite set for every \(x\in G.Y\).\footnote{Note that the action map \(G\times Y\to X\) has countable sections, and therefore by Lusin--Novikov uniformization theorem the image of \(G\times Y\) under the action map, namely the set \(G.Y\), is Borel.}
    \item For every compact set \(K\subset G\),\footnote{Using the \(G\)-invariance of \(\mu\), one can see that if this holds for one compact set in \(G\) with nonempty interior, then it holds for all compact sets in \(G\).}
    \[\int_{X}\left|Y_{x}\cap K\right|d\mu\left(x\right)<+\infty.\]
\end{enumerate}

Local integrability is nothing but the finiteness of the transverse measure:

\begin{prop}\label{prop:intensityli}
Let \(\left(X,\mu\right)\) be a probability preserving \(G\)-space with a cross section \(Y\). Then \(Y\) is locally integrable precisely when the transverse measure \(\mu_{Y}\) is finite:
\[\iota_{\mu}\left(Y\right)\coloneqq \mu_{Y}\left(Y\right)<+\infty.\]
\end{prop}

\begin{proof}
Let \(K\subset G\) be an arbitrary compact set, and look at the Borel function \(f:G\times Y\to\left[0,+\infty\right]\), \(f\left(g,y\right)\coloneqq\mathbf{1}_{K}\left(g\right)\), whose \(X\)-periodization is
\[f_{X}\left(x\right)=\sum\nolimits_{g\in Y_{x}}\mathbf{1}_{K}\left(g^{-1}\right)=\left|Y_{x}\cap K^{-1}\right|,\quad x\in X.\]
It follows from Campbell theorem~\ref{prop:campbell} that
\[m_{G}\left(K\right)\cdot\mu_{Y}\left(Y\right)=m_{G}\otimes\mu_{Y}\left(f\right)=\mu\left(f_{X}\right)=\int_{X}\left|Y_{x}\cap K^{-1}\right|d\mu\left(x\right).\]
Since this holds for any compact set \(K\), the proof follows immediately.
\end{proof}

An important family of locally integrable cross sections are \emph{separated cross sections}. For a cross section \(Y\) of a Borel \(G\)-space \(X\), define the set
\[\Lambda_{Y}\coloneqq \{g\in G:g.Y\cap Y\neq\emptyset\}=\bigcup\nolimits_{y\in Y}Y_{y}.\]
Note that \(\Lambda_{Y}\) is symmetric (namely \(\Lambda_{Y}^{-1}=\Lambda_{Y}\)), and we will refer to \(\Lambda_{Y}\) as the {\bf return times set} of \(Y\). Then a cross section \(Y\) of a Borel \(G\)-space \(X\) is said to be {\bf separated}, if \(e_{G}\in\Lambda_{Y}\) is an isolated point. The proof of the following lemma is elementary (see~\cite[Lemma 3.1(ii)]{bjorklund2025int}).

\begin{lem}\label{lem:separatedli}
If \(Y\) is a separated cross section for a Borel \(G\)-space \(X\), then \(Y\) is locally integrable with respect to every \(G\)-invariant probability measures on \(X\).
\end{lem}

\subsection{Null sets in transverse \(G\)-spaces}

The following property relates the null sets of measures to the null sets of the transverse measures. In the separated case it was proved in~\cite[Lemma 4.6]{bjorklund2025int}.

\begin{prop}\label{prop:null}
For every transverse \(G\)-space \(\left(X,\mu,Y\right)\), \(\mu\)-null and \(\mu_{Y}\)-null sets are related as follows:
\begin{enumerate}
    \item For every \(G\)-invariant Borel set \(A\subseteq X\),
    \[A\text{ is }\mu\text{-conull }\iff A\cap Y\text{ is }\mu_{Y}\text{-conull}.\]
    \item For every \(\mu_{Y}\)-conull Borel set \(B\subseteq Y\),
    \[B_{x}=Y_{x}\text{ for }\mu\text{-a.e. }x\in X,\]
    and
    \[B_{y}=Y_{y}\text{ for }\mu_{Y}\text{-a.e. }y\in Y.\]
\end{enumerate}
\end{prop}

\begin{proof}
For Part (1), suppose \(A\subseteq X\) is \(G\)-invariant. Then for every compact set \(K\subset G\), by Campbell theorem~\ref{prop:campbell} we have
\begin{align*}
m_{G}\left(K\right)\cdot\mu_{Y}\left(A\cap Y\right)	
&=\int_{X}\left(\sum\nolimits_{g\in Y_{x}}\mathbf{1}_{K}\left(g^{-1}\right)\cdot\mathbf{1}_{A\cap Y}\left(g.x\right)\right)d\mu\left(x\right)\\
&=\int_{X}\left(\sum\nolimits_{g\in Y_{x}}\mathbf{1}_{K^{-1}}\left(g\right)\cdot\mathbf{1}_{A}\left(x\right)\right)d\mu\left(x\right)=\int_{A}\left|Y_{x}\cap K^{-1}\right|d\mu\left(x\right).
\end{align*}
Now if \(\mu_{Y}\left(A\cap Y\right)=0\), then by taking a countable compact exhaustion \(K\nearrow G\) and the monotone convergence theorem we deduce that \(\int_{A}\left|Y_{x}\right|d\mu\left(x\right)=0\), hence \(Y_{x}=\emptyset\) for \(\mu\)-a.e. \(x\in A\), and this implies that \(\mu\left(A\right)=0\). Conversely, if \(\mu_{Y}\left(A\cap Y\right)>0\) then \(\int_{A}\left|Y_{x}\right|d\mu\left(x\right)>0\) hence \(\mu\left(A\right)>0\). Before proving Part (2), let us note that the proof of Part (1) holds true for every \(G\)-invariant set \(A\) which is not necessarily Borel but merely \(\mu\)-measurable.

For Part (2), suppose \(B\subseteq Y\) is \(\mu_{Y}\)-conull. Then the (analytic, hence universally measurable) set \(G.B\subseteq X\) is \(G\)-invariant and satisfies \(G.B\cap Y\supseteq B\), hence \(G.B\cap Y\) is \(\mu_{Y}\)-conull, and it follows from Part (1) that \(G.B\) is \(\mu\)-conull.  Note that \(B_{x}=Y_{x}\) for every \(x\in G.B\), implying the first property in (2). Finally, since \(A\) is \(G\)-invariant, by Part (1) the set \(A\cap Y\) is \(\mu_{Y}\)-conull, and also \(\left(A\cap Y\right)_{y}=Y_{y}\) for every \(y\in A\cap Y\), establishing the second property in (2).
\end{proof}

Using Proposition~\ref{prop:null} we may define measure equivalence of cross sections.

\begin{prop}\label{prop:crossequiv}
Let \(\left(X,\mu\right)\) be a probability preserving \(G\)-space with two cross sections \(Y,Y^{\prime}\), and thus both \(G.Y\) and \(G.Y^{\prime}\) are \(\mu\)-conull. The following are equivalent:
\begin{enumerate}
    \item \(Y_{x}=Y_{x}^{\prime}\) for \(\mu\)-a.e. \(x\in X\).
    \item \(\mu_{Y}=\mu_{Y^{\prime}}\) as measures on \(Y\cup Y^{\prime}\).
    \item \(Y\cap Y^{\prime}\) is both \(\mu_{Y}\)-conull and \(\mu_{Y^{\prime}}\)-conull.
\end{enumerate}
When these conditions hold, we consider \(Y\) and \(Y^{\prime}\) to be {\bf measure equivalent} with respect to \(\mu\).
\end{prop}

\begin{proof}
For \(1\implies 2\), let \(f:Y\cup Y^{\prime}\to\left[0,+\infty\right]\) be an arbitrary Borel function. Fix a Borel function \(w:G\to\left[0,+\infty\right]\) with \(m_{G}\left(w\right)=1\), and define
\[F:G\times Y\longrightarrow\left[0,+\infty\right],\quad F\left(g,y\right)\coloneqq w\left(g\right)\cdot f\left(y\right).\]
From the assumption it follows that the two \(X\)-periodizations of \(F\) with respect to \(Y\) and to \(Y^{\prime}\) coincide, and therefore by Campbell theorem~\ref{prop:campbell},
\[\mu_{Y}\left(f\right)=m_{G}\otimes\mu_{Y}\left(F\right)=m_{G}\otimes\mu_{Y^{\prime}}\left(F\right)=\mu_{Y^{\prime}}\left(f\right).\]
The implication \(2\implies 3\) is trivial. For \(3\implies 1\), using Proposition~\ref{prop:null}(2) twice for \(Y\) and \(Y^{\prime}\), we get
\[Y_{x}=\left(Y\cap Y^{\prime}\right)_{x}=Y_{x}^{\prime}\text{ for }\mu\text{-a.e. }x\in X.\qedhere\]
\end{proof}

We end this section by showing that the transverse correspondence~\ref{thm:corresp} commutes with the ergodic decomposition. In the separated case it was essentially proved in~\cite[Lemma~4.9]{bjorklund2025int}. We recall that by the Ergodic Decomposition Theorem, for every probability preserving \(G\)-space \(\left(X,\mu\right)\), there exists a standard probability space \(\left(\Upsilon,\upsilon\right)\), each point \(e\in\Upsilon\) is a \(G\)-ergodic Borel probability measure on \(X\), such that, for every Borel function \(f:X\to\left[0,+\infty\right]\),
\[\mu\left(f\right)=\int_{\Upsilon}e\left(f\right)d\upsilon\left(e\right).\]
We also recall that, using the notion of invariance of sets to countable Borel equivalence relations, ergodicity and the ergodic decomposition theorem apply naturally to countable Borel equivalence relations.

\begin{prop}
Let \(\left(X,\mu,Y\right)\) be a transverse \(G\)-space. Then \(\mu=\int_{\Upsilon}e d\upsilon\left(e\right)\) is the ergodic decomposition of  \(\mu\) if and only if \(\mu_{Y}=\int_{\Upsilon}e_{Y}d\upsilon\left(e\right)\) is the ergodic decomposition of \(\mu_{Y}\).
\end{prop}

\begin{proof}
By Definition~\ref{dfn:transdef} of the transverse measure, if a measure \(\mu\in\mathcal{M}^{G}\left(X\right)\) is decomposed into \(\mu=\int_{\Upsilon}ed\upsilon\left(e\right)\) (whether the components are ergodic or not), then \(\mu_{Y}=\int_{\Upsilon}e_{Y}d\upsilon\left(e\right)\). Similarly, by the inverse transverse correspondence \eqref{eq:invcorr}, if a measure \(\nu\in\mathcal{M}^{E_{G}^{Y}}\left(Y\right)\) is decomposed into \(\nu=\int_{\Upsilon}ed\upsilon\left(e\right)\) (whether the components are ergodic or not), then \(\nu^{X}=\int_{\Upsilon}e^{X}d\upsilon\left(e\right)\). Therefore, it suffices to show that \(\mu\in\mathcal{M}^{G}\left(X\right)\) is \(G\)-ergodic if and only if \(\mu_{Y}\in\mathcal{M}^{E_{G}^{Y}}\left(Y\right)\) is \(E_{G}^{Y}\)-ergodic.

Note that every \(G\)-invariant Borel set \(A\subseteq X\) gives rise to the \(E_{G}^{Y}\)-invariant Borel set \(A\cap Y\), so that from Proposition~\ref{prop:null}(1) it follows directly that if \(\mu_{Y}\) is \(E_{G}^{Y}\)-ergodic then \(\mu\) is \(G\)-ergodic. To establish the converse, let us show that the association \(A\mapsto A\cap Y\) can be reversed; that is, every \(E_{G}^{Y}\)-invariant Borel set \(A\subseteq Y\) gives rise to a \(G\)-invariant Borel set \(A_{\lambda}\subseteq X\) such that \(A_{\lambda}\cap Y=A\). Once we will show that, from Proposition~\ref{prop:null}(1) it will directly follow that if \(\mu\) is \(G\)-ergodic then \(\mu_{Y}\) is \(E_{G}^{Y}\)-ergodic. In order to show this, pick a Borel section \(\lambda:X\to G\) subject to the property \(\lambda\left(x\right)\in Y_{x}\) (using Lusin--Novikov theorem~\ref{thm:lusinnovi}), and modify it so that \(\lambda\left(y\right)=e_{G}\) for every \(y\in Y\). Now given a Borel set \(A\subseteq Y\), define
\[A_{\lambda}\coloneqq\left\{x\in X:\lambda\left(x\right).x\in A\right\} \subseteq X.\]
By the choice of \(\lambda\) it is clear that \(A_{\lambda}\cap Y=A\). When \(A\) is \(E_{G}^{Y}\)-invariant, then for every \(x\in A_{\lambda}\) and \(g\in G\) we have \(\left(\lambda\left(x\right).x,\lambda\left(g.x\right)g.x\right)\in E_{G}^{Y}\), and therefore \(\lambda\left(x\right).x\in A\iff\lambda\left(g.x\right)g.x\in A\), namely \(x\in A_{\lambda}\iff g.x\in A_{\lambda}\), and thus \(A_{\lambda}\) is \(G\)-invariant as required.
\end{proof}

\subsection{Transverse \(G\)-factors}\label{sct:transfact}

In order to define \(G\)-factors for transverse \(G\)-spaces, one must proceed carefully: cross sections are typically null sets with respect to the ambient measure, while \(G\)-factors are generally defined only up to null sets. Let us start by introducing some convenient terminology.

\begin{defn}
Let \(\left(X,\mu\right)\) and \(\left(W,\nu\right)\) be probability preserving \(G\)-spaces.
\begin{itemize}
    \item A \(G\) {\bf-factor} \(\phi:\left(X,\mu\right)\to\left(W,\nu\right)\) is a Borel map \(\phi:X\to W\), such that \(\phi_{\ast}\mu=\nu\) and
    \[\phi\left(g.x\right)=g.\phi\left(x\right)\text{ for every }g\in G\text{ and }\mu\text{-a.e. }x\in X.\footnote{Crucially, the \(\mu\)-conull set of \(x\)'s on which this property holds may depend on \(g\).}\]
    \item A {\bf concrete \(G\)-factor} \(\phi_{o}:\left(X,\mu\right)\to\left(W,\nu\right)\) is a Borel map \(\phi_{o}:X_{o}\to W\) defined on a \(\mu\)-conull \(G\)-invariant Borel set \(X_{o}\subseteq X\), such that \(\phi_{o\ast}\mu=\nu\) and
    \[\phi_{o}\left(g.x\right)=g.\phi_{o}\left(x\right)\text{ for every }\left(g,x\right)\in G\times X_{o}.\]
    \item A concrete \(G\)-factor \(\phi_{o}\) is a {\bf concrete version} of a \(G\)-factor \(\phi\), if \(\phi=\phi_{o}\) on a \(\mu\)-conull set.
\end{itemize}
\end{defn}

\begin{prop}[Folklore]\label{prop:pwfactor}
Whenever \(G\) is an lcsc group, every \(G\)-factor admits a concrete version.
\end{prop}

\begin{proof}
Let \(\phi:\left(X,\mu\right)\to\left(W,\nu\right)\) be a \(G\)-factor, and consider the Borel set
\[\Sigma\coloneqq \{\left(g,x\right)\in G\times X:\phi\left(g.x\right)=g.\phi \left(x\right)\}.\]
For \(x\in X\), consider its \(x\)-section, \(\Sigma^{x}=\left\{g\in G: \phi\left(g.x\right)=g.\phi\left(x\right)\right\} \), and by Fubini Theorem with the measure \(m_{G}\otimes\mu\), the set
\[X^{o}\coloneqq \left\{x\in X:\Sigma^{x}\text{ is }m_{G}\text{-conull}\right\}\]
is \(\mu\)-conull. We claim that
\begin{equation}\label{eq:phi}
\phi\left(h.x\right)=h.\phi\left(x\right)\text{ whenever }\left(h,x\right)\in G\times X^{o}\text{ satisfies }h.x\in X^{o}.
\end{equation}
Indeed, by the assumption both sets \(\Sigma^{h.x}\) and \(\Sigma^{x}h^{-1}\) are \(m_{G}\)-conull, and therefore their intersection is \(m_{G}\)-conull, and in particular nonempty. Then pick any \(g\in \Sigma^{h.x}\cap \Sigma^{x}h^{-1}\), and we have
\[g.\phi\left(h.x\right)=\phi\left(gh.x\right)=gh.\phi\left(x\right)\text{, hence }\phi\left(h.x\right)=h.\phi\left(x\right),\]
establishing \eqref{eq:phi}. Using~\cite[Lemma B.8]{Zimmer1984},\footnote{This lemma follows from Varadarajan's compact model theorem~\cite[Ch.~V, \S3, Theorem~5.7]{Varadarajan1968}, with a theorem attributed to Kallman~\cite[Theorem A.5]{Zimmer1984}, which is the uniformization theorem~\cite[Theorem (18.18)]{kechris2012classical}.} there is a Borel set \(X_{o}^{o}\subseteq X^{o}\) such that \(\mu\left(X^{o}\setminus X_{o}^{o}\right)=0\) and \(G.X_{o}^{o}\) is Borel, as well as a Borel map \(\lambda:G.X_{o}^{o}\to G\) such that \(\lambda\left(x\right).x\in X_{o}^{o}\) for every \(x\in G.X_{o}^{o}\) and \(\lambda\left(x\right)=e_{G}\) for every \(x\in X_{o}^{o}\). Put the Borel set
\[X_{o}\coloneqq G.X_{o}^{o},\]
and define the Borel map
\[\phi_{o}:X_{o}\longrightarrow W,\quad \phi_{o}\left(x\right)\coloneqq \lambda\left(x\right)^{-1}.\phi\left(\lambda\left(x\right).x\right),\quad x\in X_{o}.\]
Since \(X^{o}\) is \(\mu\)-conull then so is \(X_{o}^{o}\) and a fortiori \(X_{o}\). Since \(\lambda\mid_{X_{o}^{o}}\equiv e_{G}\) and \(X_{o}^{o}\subseteq X_{o}\) is \(\mu\)-conull, we have \(\phi_{o}\left(x\right)=\phi\left(x\right)\) for \(\mu\)-a.e. \(x\in X_{o}\). To see the equivariance of \(\phi_{o}\), let \(\left(g,x\right)\in G\times X_{o}\), and note that
\[\big(\lambda\left(g.x\right).g.\lambda\left(x\right)^{-1},\lambda\left(x\right).x\big)\in G\times X_{o}\subseteq G\times X^{o}\text{ satisfies }\lambda\left(g.x\right).g.\lambda\left(x\right)^{-1}.\lambda\left(x\right).x=\lambda\left(g.x\right).g.x\in X_{o}\subseteq X^{o},\]
and therefore, using \eqref{eq:phi} we obtain
\begin{align*}
\phi_{o}\left(g.x\right)	
&=\lambda\left(g.x\right)^{-1}.\phi\left(\lambda\left(g.x\right).g.x\right)\\
&=\lambda\left(g.x\right)^{-1}.\phi\big(\small[\lambda\left(g.x\right).g.\lambda\left(x\right)^{-1}\small].\left[\lambda\left(x\right).x\right]\big)\\
&=\lambda\left(g.x\right)^{-1}.\small[\lambda\left(g.x\right).g.\lambda\left(x\right)^{-1}\small].\phi\left(\lambda\left(x\right).x\right)\\
&=g.\lambda\left(x\right)^{-1}.\phi\left(\lambda\left(x\right).x\right)\\
&=g.\phi_{o}\left(x\right).\qedhere
\end{align*}
\end{proof}

We now define \(G\)-factors in the context of transverse \(G\)-spaces.

\begin{defn}
Let \(\left(X,\mu,Y\right)\) and \(\left(W,\nu,Z\right)\) be transverse \(G\)-spaces.
\begin{itemize}
    \item A {\bf transverse \(G\)-factor} \(\phi:\left(X,\mu,Y\right)\to\left(W,\nu,Z\right)\) is a \(G\)-factor \(\phi:\left(X,\mu\right)\to\left(W,\nu\right)\) such that
    \[Y_{x}=Z_{\phi\left(x\right)}\text{ for }\mu\text{-a.e. }x\in X.\]
    \item A {\bf concrete transverse \(G\)-factor} \(\phi_{o}:\left(X,\mu,Y\right)\to\left(W,\nu,Z\right)\) is a concrete \(G\)-factor \(\phi_{o}:\left(X,\mu\right)\to\left(W,\nu\right)\), defined on a \(\mu\)-conull \(G\)-invariant Borel set \(X_{o}\subseteq X\), and a cross section \(Y_{o}\subseteq X_{o}\), such that
    \[\left(Y_{o}\right)_{x}=Z_{\phi_{o}\left(x\right)}\text{ for every }x\in X_{o}.\]
    Observe that this property is equivalent to that
    \[Y_{o}=\phi_{o}^{-1}\left(Z\right).\]
    \item A concrete transverse \(G\)-factor \(\phi_{o}\) is a {\bf concrete version} of a transverse \(G\)-factor \(\phi\), if \(\phi=\phi_{o}\) on a \(\mu\)-conull set and \(Y_{o}\) is measure equivalent to \(Y\) (as in Proposition~\ref{prop:crossequiv}).
\end{itemize}
\end{defn}

\begin{prop}\label{prop:conctrans}
Every transverse \(G\)-factor admits a concrete version.
\end{prop}

\begin{proof}
Let \(\phi:\left(X,\mu,Y\right)\to\left(W,\nu,Z\right)\) be a transverse \(G\)-factor. Using Proposition~\ref{prop:pwfactor}, pick a concrete version \(\phi_{o}\) of \(\phi\) as a mere \(G\)-factor, defined on a \(\mu\)-conull \(G\)-invariant set \(X_{o}\subseteq X\). Put \(Y_{o}\coloneqq{\phi_{o}}^{-1}\left(Z\right)\cap X_{o}\), and from the concreteness of \(\phi_{o}\) and the \(G\)-invariance of \(X_{o}\), it follows that \(\left(Y_{o}\right)_{x}=Z_{\phi_{o}\left(x\right)}\) for every \(x\in X_{o}\). In particular, since \(\phi\) is a transverse \(G\)-factor, it follows that \(\left(Y_{o}\right)_{x}=Z_{\phi_{o}\left(z\right)}=Z_{\phi\left(x\right)}=Y_{x}\) for \(\mu\)-a.e. \(x\in X\), and thus \(Y_{o}\) is measure equivalent to \(Y\). Therefore, \(\phi_{o}\) is a concrete version of \(\phi\) as a transverse \(G\)-factor.
\end{proof}

We now have a property that can be seen as functoriality of the transverse measure construction.

\begin{prop}\label{prop:factor}
Let \(\left(X,\mu,Y\right)\) and \(\left(W,\nu,Z\right)\) be transverse \(G\)-spaces. For every concrete transverse \(G\)-factor \(\phi_{o}:\left(X,\mu,Y\right)\to \left(W,\nu,Z\right)\) it holds that
\[\phi_{o\ast}\mu_{Y}=\nu_{Z}.\]
Therefore, if there exists a transverse \(G\)-factor \(\left(X,\mu,Y\right)\to \left(W,\nu,Z\right)\) then
\[\iota_{\mu}\left(Y\right)=\iota_{\nu}\left(Z\right).\] 
\end{prop}

\begin{proof}
Let \(\phi_{o}:X_{o}\to W\) be a concrete transverse \(G\)-factor, defined on a \(\mu\)-conull \(G\)-invariant Borel set \(X_{o}\subseteq X\) and a cross section \(Y_{o}\) measure equivalent to \(Y\). Fix a Borel function \(w:G\to\left[0,+\infty\right]\) with \(m_{G}\left(w\right)=1\). For an arbitrary Borel function, \(f:Z\to\left[0,+\infty\right]\), define the Borel functions
\[F:G\times Z\longrightarrow\left[0,+\infty\right],\quad F\left(g,z\right)\coloneqq w\left(g\right)\cdot f\left(z\right),\]
and
\[F^{\prime}:G\times Y\to\left[0,+\infty\right],\quad F^{\prime}\left(g,y\right)\coloneqq F\left(g,\phi\left(y\right)\right)=w\left(g\right)\cdot f\left(\phi\left(y\right)\right).\]
Then since \(\phi_{o}\) is concrete, for every \(x\in X_{o}\) it holds that
\[F_{X}^{\prime}\left(x\right)=\sum\nolimits_{g\in Y_{x}}w\left(g^{-1}\right)\cdot F\left(g.\phi_{o}\left(x\right)\right)=\sum\nolimits_{g\in Z_{\phi_{o}\left(x\right)}}w\left(g^{-1}\right)\cdot F\left(g.\phi_{o}\left(x\right)\right)=F_{W}\left(\phi_{o}\left(x\right)\right).\]
Applying Campbell theorem~\ref{prop:campbell} twice and using \(m_{G}\left(w\right)=1\), we obtain that
\[\phi_{o\ast}\mu_{Y_{o}}\left(f\right)=m_{G}\otimes\mu_{Y_{o}}\left(F^{\prime}\right)=\mu\big(F_{X}^{\prime}\big)=\phi_{o\ast}\mu\left(F_{W}\right)=\nu\left(F_{W}\right)=m_{G}\otimes\nu_{Z}\left(F\right)=\nu_{Z}\left(f\right).\]
Since \(Y_{o}\) and \(Y\) are measure equivalent, \(\mu_{Y}=\mu_{Y_{o}}\) and therefore \(\phi_{o\ast}\mu_{Y}=\phi_{o\ast}\mu_{Y_{o}}=\nu_{Z}\).

Finally, if there exists a transverse \(G\)-factor \(\left(X,\mu,Y\right)\to \left(W,\nu,Z\right)\), then by Proposition~\ref{prop:conctrans} there exists a concrete transverse \(G\)-factor \(\phi_{o}:\left(X,\mu,Y\right)\to\left(W,\nu,Z\right)\), and thus
\[\iota_{\nu}\left(Z\right)	=\nu_{Z}\left(Z\right)=\phi_{o\ast}\mu_{Y}\left(Z\right)=\phi_{o\ast}\mu_{Y_{o}}\left(Z\right)=\mu_{Y_{o}}\left(\phi_{o}^{-1}\left(Z\right)\right)=\mu_{Y_{o}}\left(Y_{o}\right)=\mu_{Y}\left(Y\right)=\iota_{\mu}\left(Y\right).\qedhere\]
\end{proof}

\section{Intersection spaces}\label{sct:intsp}

Intersections spaces were introduced by the second author, Hartnick and Karasik~\cite{bjorklund2025int} for separated cross sections. Using the transverse correspondence~\ref{thm:corresp}, it extends to transverse \(G\)-spaces as follows.

Let \(r\in\mathbb{N}\) be arbitrary and fixed, and let \(X\) be a Borel \(G\)-space with a cross section \(Y\). Denote by \(X^{\otimes r}\) and \(Y^{\otimes r}\) the \(r\)-fold product of \(X\) and \(Y\), respectively, and consider the diagonal action of \(G\) on \(X^{\otimes r}\), that we denote by
\[g\diag\left(x_{1},\dotsc,x_{r}\right)=\left(g.x_{1},\dotsc,g.x_{r}\right).\]
The {\bf intersection space} of order \(r\) associated with \(X\) is the Borel \(G\)-space
\[Y^{\left[r\right]}\coloneqq G\diag Y^{\otimes r}=\{\left(g.y_{1},\dotsc,g.y_{r}\right):g\in G,\,\left(y_{1},\dotsc,y_{r}\right)\in Y^{\otimes r}\}\subseteq X^{\otimes r}.\]
Being a \(G\)-invariant subsets of the Borel \(G\)-space \(X^{\otimes r}\) in the diagonal action of \(G\), the intersection space \(Y^{\left[r\right]}\) becomes a Borel \(G\)-spaces on its own right. Furthermore, observe that it admits a natural cross section, namely
\[Y^{\otimes r}\subseteq Y^{\left[r\right]},\]
whose return times sets are of the form
\[Y_{g\diag\left(y_{1},\dotsc,y_{r}\right)}^{\otimes r}=Y_{g.y_{1}}\cap\dotsm\cap Y_{g.y_{r}}=\left(Y_{y_{1}}\cap\dotsm\cap Y_{y_{r}}\right)g^{-1}.\]

\begin{defn}\label{def:intspace1}
The {\bf intersection space} of order \(r\in\mathbb{N}\) of a transverse \(G\)-space \(\left(X,\mu,Y\right)\), is the measure preserving \(G\)-space
\[\big(Y^{\left[r\right]},\mu^{\left[r\right]}\big),\]
where \(\mu^{\left[r\right]}\) is specified by that the transverse measure associated with \(\left(Y^{\left[r\right]},\mu^{\left[r\right]},Y^{\otimes r}\right)\) is \(\mu_{Y}^{\otimes r}\); in the notations of the transverse correspondence~\ref{thm:corresp} or its inverse \eqref{eq:invcorr},
\[\mu^{\left[r\right]}_{Y^{\otimes r}}=\mu_{Y}^{\otimes r}\text{ or }\mu^{\left[r\right]}=\left(\mu_{Y}^{\otimes r}\right)^{{\scriptscriptstyle Y^{\left[r\right]}}}.\]
\end{defn}

To verify that \(\mu^{\left[r\right]}\in\mathcal{M}^{G}\left(Y^{\left[r\right]}\right)\) is well-defined, observe that from \(\mu_{Y}\in\mathcal{M}^{E_{G}^{Y}}\left(Y\right)\) it follows that \(\mu_{Y}^{\otimes r}\in\mathcal{M}^{E_{G}^{Y^{\otimes r}}}\left(Y^{\otimes r}\right)\), where \(E_{G}^{Y^{\otimes r}}\coloneqq E_{G}^{X^{\otimes r}}\cap\left(Y^{\otimes r}\times Y^{\otimes r}\right)\) in the diagonal action of \(G\), and therefore the transverse correspondence~\ref{thm:corresp} is applicable.

We now give another description of intersection spaces for cross section satisfying some extra separation property. This will be particularly useful in proving Theorem~\ref{thm:finiteiv} as well as in our subsequent work~\cite{AvBjCuII}. This characterization was proved in~\cite[\S8.3]{bjorklund2025int} in the separated case.

Let \(\left(X,\mu,Y\right)\) be a transverse \(G\)-space, and assume it satisfies the following condition:
\begin{equation}\label{eq:doubs}\tag{\(\bigstar\)}
Y_{x_{1}}^{-1}Y_{x_{2}}\text{ is locally finite for all }x_{1},x_{2}\in X.
\end{equation}
Observe that a sufficient condition for \eqref{eq:doubs} to hold is that \(\Lambda_{Y}^{2}\) is locally finite, since for all \(x_{1},x_{2}\in X\) one has that \(Y_{x_{1}}^{-1}Y_{x_{2}}\subseteq g_{1}\Lambda_{Y}^{2}g_{2}^{-1}\)  for any choice of \(g_{1}\in Y_{x_{1}}\) and \(g_{2}\in Y_{x_{2}}\). In particular when \(\Lambda_{Y}\) is uniformly discrete, which is equivalent to that \(e_{G}\in\Lambda_{Y}^{2}\) is an isolated point.

For \(r\in\mathbb{N}\), define the Borel \(G\)-space
\[\overline{Y}^{\left[r\right]}\coloneqq X\times Y^{\left[r-1\right]},\]
(note that \(\overline{Y}^{\left[2\right]}=X^{\otimes 2}\)), where \(G\) acts only on the first coordinate:
\[g.\left(x,\vec{z}\right)=\left(g.x,\vec{z}\right).\]
Then \(Y^{\left[r\right]}\) is a cross section for \(\overline{Y}^{\left[r\right]}\) in this action, and the reader may directly verify that the return times sets are of the form
\[Y_{\left(x,\vec{z}\right)}^{\left[r\right]}=\left(Y_{z_{r}}^{-1}\cap\dotsm\cap Y_{z_{2}}^{-1}\right)Y_{x}\text{ for all }\left(x,\vec{z}\right)=\left(x,z_{2},\dotsc,z_{r}\right)\in\overline{Y}^{\left[r\right]}.\]
Define on \(\overline{Y}^{\left[r\right]}\) the measure
\[\overline{\mu}^{\left[r\right]}\coloneqq\mu\otimes\mu^{\left[r-1\right]},\]
(note that \(\overline{\mu}^{\left[2\right]}=\mu^{\otimes 2}\)), and it is clear that \(\overline{\mu}^{\left[r\right]}\) is invariant under the given action. The following identity was proved in the separated case as part of~\cite[Theorem 7.2]{bjorklund2025int}.

\begin{prop}\label{prop:intmeasiden}
For every transverse \(G\)-space \(\left(X,\mu,Y\right)\) satisfying \eqref{eq:doubs}, the following identity holds:
\[\mu^{\left[r\right]}=\overline{\mu}_{Y^{\left[r\right]}}^{\left[r\right]};\]
that is, the intersection measure \(\mu^{\left[r\right]}\) is the transverse measure of the transverse \(G\)-space \(\big(\overline{Y}^{\left[r\right]},\overline{\mu}^{\left[r\right]},Y^{\left[r\right]}\big)\).
\end{prop}

\begin{figure}[h]
\captionsetup{font={scriptsize,stretch=1}}
\[
\xymatrix@C=3.8em{
&  & \big(\overline{Y}^{\left[r\right]},\overline{\mu}^{\left[r\right]},Y^{\left[r\right]}\big)\ar@{<-->}[d]\\
\big(Y^{\left[r\right]},\left(\mu_{Y}^{\otimes r}\right)^{Y^{\left[r\right]}},Y^{\otimes r}\big)\ar@{<-->}[d]\ar@{=}[r]^{\qquad\qquad\mathrm{Def}\,\ref{def:intspace1}} & \mu^{\left[r\right]}\ar@{=}[r]^{\qquad\quad\mathrm{Prop}\,\ref{prop:intmeasiden}\qquad\qquad\quad} & \big(Y^{\left[r\right]},\overline{\mu}^{\left[r\right]}_{Y^{\left[r\right]}}\big)\\
\left(Y^{\otimes r},\mu_{Y}^{\otimes r}\right)
}
\]
\caption[.]{The dashed arrows are formed by the transverse correspondence~\ref{thm:corresp}, with the upper space being transverse \(G\)-space and the lower being the associated transverse measure.}
\label{fig:intmeas}
\end{figure}

\begin{proof}
The key identity we need is that for every Borel function \(f:G\times Y^{\otimes r}\to\left[0,+\infty\right]\),
\begin{equation}\label{eq:transiden}
m_{G}\otimes\mu_{Y}^{\otimes r}\left(f\right)=\overline{\mu}_{Y^{\left[r\right]}}^{\left[r\right]}\left(f_{Y^{\left[r\right]}}\right),
\end{equation}
where \(f_{Y^{\left[r\right]}}\) denotes the \(Y^{\left[r\right]}\)-periodization of \(f\) with respect to the cross section \(Y^{\otimes r}\), namely
\[f_{Y^{\left[r\right]}}\left(\vec{z}\right)=\sum\nolimits_{h\in Y_{z_{1}}\cap\dotsm\cap Y_{z_{r}}}f\left(h^{-1},h\diag\vec{z}\right),\quad\vec{z}\in Y^{\left[r\right]}.\]
Once we establish \eqref{eq:transiden}, then in light of Campbell theorem~\ref{prop:campbell} we can deduce that \(\overline{\mu}_{Y^{\left[r\right]}}^{\left[r\right]}\) satisfies the property \(\big(\overline{\mu}_{Y^{\left[r\right]}}^{\left[r\right]}\big)_{Y^{\otimes r}}=\mu_{Y}^{\otimes r}\), which is the defining property of \(\mu^{\left[r\right]}\), and thus \(\overline{\mu}_{Y^{\left[r\right]}}^{\left[r\right]}=\mu^{\left[r\right]}\). Therefore, for the rest of the proof we will establish the identity \eqref{eq:transiden}.

Fix a Borel function \(w:G\to\left[0,+\infty\right]\) with \(m_{G}\left(w\right)=1\). For every \(\left(x,\vec{z}\right)=\left(x,z_{2},\dotsc,z_{r}\right)\in\overline{Y}^{\left[r\right]}\), we note the following identity:
\begin{equation}\label{eq:mididen}
\begin{aligned}
&\sum\nolimits_{g\in Y_{\left(x,\vec{z}\right)}^{\left[r\right]}}f_{Y^{\left[r\right]}}\left(g.x,z_{2},\dotsc,z_{r}\right)w\left(g\right)\\
&=\sum\nolimits_{g\in\left(Y_{z_{r}}^{-1}\cap\dotsm\cap Y_{z_{2}}^{-1}\right)Y_{x}}\sum\nolimits_{h\in Y_{g.x}\cap Y_{z_{2}}\cap\dotsm\cap Y_{z_{r}}}f\left(h^{-1},h\diag\left(g.x,z_{2},\dotsc,z_{r}\right)\right)w\left(g\right)\\
\left(\mathrm{i}\right)&=\sum\nolimits_{g_{1}\in Y_{x}}\sum\nolimits_{g_{2}\in Y_{z_{2}}\cap\dotsm\cap Y_{z_{r}}}f\left(g_{2}^{-1},\left(g_{1}.x,g_{2}\diag\vec{z}\right)\right)w\left(g_{2}^{-1}g_{1}\right)\\
\left(\mathrm{ii}\right)&\coloneqq\sum\nolimits_{g_{1}\in Y_{x}}f_{Y^{\left[r-1\right]}}^{\left(g_{1},x\right)}\left(\vec{z}\right),
\end{aligned}
\end{equation}
where \(\left(\mathrm{i}\right)\) holds since \(\left(g_{1},g_{2}\right)\mapsto\left(g_{2}g_{1},g_{2}\right)\) forms a bijection of the set
\[\left\{\left(g_{1},g_{2}\right)\in G\times G:g_{1}\in\left(Y_{z_{r}}^{-1}\cap\dotsm\cap Y_{z_{2}}^{-1}\right)Y_{x},\,g_{2}\in Y_{g_{1}.x}\cap Y_{z_{2}}\cap\dotsm\cap Y_{z_{r}}\right\}\]
onto the set
\[Y_{x}\times\left(Y_{z_{2}}\cap\dotsm\cap Y_{z_{r}}\right);\]
and in \(\left(\mathrm{ii}\right)\) we have defined the family of Borel functions with parameter \(\left(g_{1},x\right)\), \(g_{1}\in Y_{x}\), by
\[f^{\left(g_{1},x\right)}:G\times Y^{\otimes\left(r-1\right)}\longrightarrow\left[0,+\infty\right],\quad f^{\left(g_{1},x\right)}\left(g_{2},\vec{y}\right)\coloneqq f\left(g_{2},\left(g_{1}.x,\vec{y}\right)\right)w\left(g_{2}g_{1}\right),\]
whose \(Y^{\left[r-1\right]}\)-periodization with respect to the cross section \(Y^{\otimes\left(r-1\right)}\) of the Borel \(G\)-space \(Y^{\left[r-1\right]}\) with the diagonal action is indeed
\begin{align*}
f_{Y^{\left[r-1\right]}}^{\left(g_{1},x\right)}\left(\vec{z}\right)	
&=\sum\nolimits_{g_{2}\in Y_{z_{2}}\cap\dotsm\cap Y_{z_{r}}}f^{\left(g_{1},x\right)}\left(g_{2}^{-1},g_{2}\diag\vec{z}\right)\\
&=\sum\nolimits_{g_{2}\in Y_{z_{2}}\cap\dotsm\cap Y_{z_{r}}}f\left(g_{2}^{-1},\left(g_{1}.x,g_{2}\diag\vec{z}\right)\right)w\left(g_{2}^{-1}g_{1}\right).
\end{align*}
Now by Definition~\ref{dfn:transdef} of the transverse measure for the transverse \(G\)-space \(\big(\overline{Y}^{\left[r\right]},\overline{\mu}^{\left[r\right]},Y^{\left[r\right]}\big)\), and using \eqref{eq:mididen} and Fubini's theorem for \(\overline{\mu}^{\left[r\right]}=\mu\otimes\mu^{\left[r-1\right]}\), we obtain
\begin{align*}
\overline{\mu}_{Y^{\left[r\right]}}^{\left[r\right]}\big(f_{Y^{\left[r\right]}}\big)
&=\int_{\overline{Y}^{\left[r\right]}}\sum\nolimits_{g\in Y_{\left(x,\vec{z}\right)}^{\left[r\right]}}f_{Y^{\left[r\right]}}\left(g.x,\vec{z}\right)w\left(g\right)d\overline{\mu}^{\left[r\right]}\left(x,\vec{z}\right)\\
&=\int_{\overline{Y}^{\left[r\right]}}\sum\nolimits_{g_{1}\in Y_{x}}f_{Y^{\left[r-1\right]}}^{\left(g_{1},x\right)}\left(\vec{z}\right)d\overline{\mu}^{\left[r\right]}\left(x,\vec{z}\right)\\
&=\int_{X}\sum\nolimits_{g_{1}\in Y_{x}}\Big[\int_{Y^{\left[r-1\right]}}f_{Y^{\left[r-1\right]}}^{\left(g_{1},x\right)}\left(\vec{z}\right)d\mu^{\left[r-1\right]}\left(\vec{z}\right)\Big]d\mu\left(x\right).
\end{align*}
For every \(x\in X\) and \(g_{1}\in Y_{x}\), using Campbell theorem~\ref{prop:campbell} and that \(\mu_{Y^{\otimes\left(r-1\right)}}^{\left[r-1\right]}=\mu_{Y}^{\otimes\left(r-1\right)}\), we get
\[\int_{Y^{\left[r-1\right]}}f_{Y^{\left[r-1\right]}}^{\left(g_{1},x\right)}\left(\vec{z}\right)d\mu^{\left[r-1\right]}\left(\vec{z}\right)=m_{G}\otimes\mu_{Y}^{\otimes\left(r-1\right)}\big(f^{\left(g_{1},x\right)}\big),\]
and therefore, summing over \(g_{1}\in Y_{x}\) and integrating against \(d\mu\left(x\right)\), using the definition of \(f^{\left(g_{1},x\right)}\) as well as Fubini's theorem for \(m_{G}\otimes\mu_{Y}^{\otimes\left(r-1\right)}\otimes\mu\), we obtain
\begin{align*}
\overline{\mu}_{Y^{\left[r\right]}}^{\left[r\right]}\left(f_{Y^{\left[r\right]}}\right)
&=\int_{X}\sum\nolimits_{g_{1}\in Y_{x}}m_{G}\otimes\mu_{Y}^{\otimes\left(r-1\right)}\left(f^{\left(g_{1},x\right)}\right)d\mu\left(x\right)\\
&=\int_{G}\int_{Y^{\otimes\left(r-1\right)}}\Big[\int_{X}\sum\nolimits_{g_{1}\in Y_{x}}f\left(g_{2},\left(g_{1}.x,\vec{y}\right)\right)w\left(g_{2}g_{1}\right)d\mu\left(x\right)\Big]dm_{G}\left(g_{2}\right)d\mu_{Y}^{\otimes\left(r-1\right)}\left(\vec{y}\right)\\
\left(\ast\right)&=\int_{G}\int_{Y^{\otimes\left(r-1\right)}}\Big[\int_{Y}f\left(g_{2},\left(y_{1},\vec{y}\right)\right)d\mu_{Y}\left(y_{1}\right)\Big]dm_{G}\left(g_{2}\right)d\mu_{Y}^{\otimes\left(r-1\right)}\left(\vec{y}\right)\\
&=m_{G}\otimes\mu_{Y}^{\otimes r}\left(f\right),
\end{align*}
where \(\left(\ast\right)\) is by Definition~\ref{dfn:transdef} of \(\mu_{Y}\), applied to each of the functions \(Y\to\left[0,+\infty\right]\), \(y_{1}\mapsto f\left(g_{2},\left(y_{1},\vec{y}\right)\right)\), with fixed \(g_{2}\in G\) and \(\vec{y}\in Y^{\otimes\left(r-1\right)}\) (using \(\int_{G}w\left(g_{2}g_{1}\right)dm_{G}\left(g_{1}\right)=1\) for every \(g_{2}\in G\) by unimodularity). We thus obtained the desired identity \eqref{eq:transiden}, which completes the proof.
\end{proof}

\section{The intersection covolume and a higher order Kac's lemma}\label{sct:intvol}

Now we came to define the intersection covolume. Let \(G\) be an lcsc unimodular group and \(\left(X,\mu,Y\right)\) a transverse \(G\)-space. The {\bf intersection covolume} of \(\left(X,\mu,Y\right)\) is the quantities
\[I_{\mu}^{r}\left(Y\right)\coloneqq \mu^{\left[r\right]}\big(Y^{\left[r\right]}\big)\in\left(0,+\infty\right],\quad r\in\mathbb{N},\]
where \(\left(Y^{\left[r\right]},\mu^{\left[r\right]}\right)\) is the intersection space of order \(r\). Thus, \(I_{\mu}^{1}\left(Y\right)=\mu\left(X\right)=1\) and \(I_{\mu}^{2}\left(Y\right)=I_{\mu}\left(Y\right)\).

\smallskip

A basic property of the intensity and the intersection covolume is monotonicity:

\begin{prop}\label{prop:mono}
Let \(X\) be a Borel \(G\)-space with cross sections \(Y\subseteq Z\). Then for every \(\mu\in\mathcal{M}^{G}\left(X\right)\),
\[\iota_{\mu}\left(Y\right)\leq\iota_{\mu}\left(Z\right)\text{ and }I_{\mu}^{r}\left(Y\right)\leq I_{\mu}^{r}\left(Z\right)\text{ for all }r\in\mathbb{N}.\]
\end{prop}

\begin{proof}
Let \(\mu\in\mathcal{M}^{G}\left(X\right)\). Since \(Y\subseteq Z\), we have \(Y_{x}\subseteq Z_{x}\) for every \(x\in X\), and from Definition~\ref{dfn:transdef} of transverse measure it follows that \(\mu_{Y}\left(f\right)\leq\mu_{Z}\left(f\right)\) for every Borel function \(f:Y\to\left[0,+\infty\right]\), and hence
\[\iota_{\mu}\left(Y\right)=\mu_{Y}\left(\mathbf{1}_{Y}\right)\leq\mu_{Z}\left(\mathbf{1}_{Y}\right)\leq\mu_{Z}\left(\mathbf{1}_{Z}\right)=\iota_{\mu}\left(Z\right).\]
For the intersection covolume, from \(Y\subseteq Z\) we have \(Y^{\otimes r}\subseteq Z^{\otimes r}\) and \(Y^{\left[r\right]}\subseteq Z^{\left[r\right]}\), so we can apply Lemma~\ref{lem:borpartmon}. Therefore, there are Borel partitions of unity \(\rho^{Y}\) for \(Y^{\left[r\right]}\) with respect to \(Y^{\otimes r}\) and \(\rho^{Z}\) for \(Z^{\left[r\right]}\) with respect to \(Z^{\otimes r}\), such that
\begin{equation}\label{eq:rhos}
\rho^{Y}\left(g,\vec{x}\right)=\rho^{Z}\left(g,\vec{x}\right)\text{ for every }\left(g,\vec{x}\right)\in G\times Y^{\left[r\right]}.
\end{equation}
Now with the notations of the inverse transverse correspondence in \eqref{eq:invcorr}, let
\[\mu^{\left[r,Y\right]}=\left(\mu_{Y}^{\otimes r}\right)^{Y^{\left[r\right]}}\text{ and }\mu^{\left[r,Z\right]}=\left(\mu_{Z}^{\otimes r}\right)^{Z^{\left[r\right]}}\]
be the intersection measures on \(Y^{\left[r\right]}\) and \(Z^{\left[r\right]}\). We then get the desired inequality by
\begin{align*}
I_{\mu}^{r}\left(Y\right)	=\mu^{\left[r,Y\right]}\left(\mathbf{1}_{Y^{\left[r\right]}}\right)
&=\iint\nolimits_{G\times Y^{\otimes r}}\rho^{Y}\left(g,g^{-1}.\vec{y}\right)\cdot\mathbf{1}_{Y^{\left[r\right]}}\left(g^{-1}.\vec{y}\right)dm_{G}\left(g\right)d\mu_{Y}^{\otimes r}\left(\vec{y}\right)\\
&=\iint\nolimits_{G\times Y^{\otimes r}}\rho^{Z}\left(g,g^{-1}.\vec{y}\right)\cdot\mathbf{1}_{Y^{\left[r\right]}}\left(g^{-1}.\vec{y}\right)dm_{G}\left(g\right)d\mu_{Y}^{\otimes r}\left(\vec{y}\right)\\
&\leq\iint\nolimits_{G\times Y^{\otimes r}}\rho^{Z}\left(g,g^{-1}.\vec{y}\right)\cdot\mathbf{1}_{Y^{\left[r\right]}}\left(g^{-1}.\vec{y}\right)dm_{G}\left(g\right)d\mu_{Z}^{\otimes r}\left(\vec{y}\right)\\
&\leq\iint\nolimits_{G\times Z^{\otimes r}}\rho^{Z}\left(g,g^{-1}.\vec{z}\right)\cdot\mathbf{1}_{Z^{\left[r\right]}}\left(g^{-1}.\vec{z}\right)dm_{G}\left(g\right)d\mu_{Z}^{\otimes r}\left(\vec{z}\right)\\
&=\mu^{\left[r,Y\right]}\left(\mathbf{1}_{Z^{\left[r\right]}}\right)=I_{\mu}^{r}\left(Z\right),
\end{align*}
where the second equality by definition of the inverse transverse correspondence \eqref{eq:invcorr}; the third equality is by \eqref{eq:rhos}; the first inequality is since \(\mu_{Y}\leq\mu_{Z}\) on \(Y\), so that \(\mu_{Y}^{\otimes r}\leq\mu_{Z}^{\otimes r}\) on \(Y^{\otimes r}\); and, the second inequality is since \(\rho^{Z}\) is nonnegative and \(\mathbf{1}_{Y^{\left[r\right]}}\leq\mathbf{1}_{Z^{\left[r\right]}}\) everywhere.
\end{proof}

In the following we establish a formula for the intersection covolume which extends the classical Kac's lemma in the case of \(G=\mathbb{R}\) and \(r=1\) (cf.~\cite{meyerovitch2024}). Note that Kac's lemma uses the natural Voronoi tesselation of \(\mathbb{R}\) induced by a locally finite set \(P\subset\mathbb{R}\), namely the partition of \(\mathbb{R}\) into left-closed right-open intervals whose middle points are the points of \(P\). This Voronoi tessellation scheme of \(\mathbb{R}\) is equivariant, in the sense that the tessellation induced by a \(t\)-translation of \(P\) is the \(t\)-translation of the tessellation induced by \(P\). Thus, we first need to find an equivariant Voronoi tesselation scheme for a general \(G\).

\subsection{Equivariant Voronoi tessellation scheme}\label{sct:voronoi}

Recall that a set \(P\subset G\) is {\bf locally finite} if its intersection with every compact set in \(G\) in finite. By Struble's theorem~\cite{Struble1974}, every lcsc group \(G\) admits a compatible metric, that is a right-invariant proper metric \(d\) which induces \(G\)'s topology. The following proposition was observed by Ab\'{e}rt and Mellick~\cite[\S3.3]{abert2022point}.

\begin{prop}\label{prop:voronoi}
Every lcsc group \(G\) with a compatible metric \(d\), admits a Voronoi tessellation scheme,
\[\mathscr{T}:P\mapsto\mathscr{T}\left(P\right)\coloneqq \left\{\Theta\left(P,p\right): p\in P\right\},\]
associating with every locally finite set \(P\subset G\) a Borel partition \(\mathscr{T}\left(P\right)\) of \(G\), such that:
\begin{enumerate}
    \item For every \(p\in P\), the cell \(\Theta\left(P,p\right)\subseteq G\) consists of points whose \(d\)-nearest point from \(P\) is \(p\).\footnote{Importantly, a point in \(\Theta\left(P,p\right)\) may have multiple nearest points from \(P\), and one of them is \(p\).} In particular, \(\Theta\left(P,p\right)\) has nonempty interior.
    \item If \(P\subseteq Q\) are locally finite sets, then \(\Theta\left(P,p\right)\supseteq\Theta\left(Q,p\right)\) for all \(p\in P\). If additionally \(\Theta\left(P,p\right)=\Theta\left(Q,p\right)\) modulo \(m_{G}\) for all \(p\in P\), then necessarily \(P=Q\).
    \item \(\mathscr{T}\) is equivariant: for every locally finite set \(P\) and \(g\in G\),
    \[\mathscr{T}\left(Pg\right)=\mathscr{T}\left(P\right)g.\]
    That is, \(\Theta\left(Pg,pg\right)=\Theta\left(P,p\right)g\) for every \(p\in P\).
\end{enumerate}
\end{prop}

\begin{rem}
The Voronoi tessellation scheme constructed in Proposition~\ref{prop:voronoi} is highly non-canonical. Note that when \(d\)-spheres in \(G\) have zero Haar measure, the Voronoi tessellation is canonical modulo Haar measure, as long as the metric \(d\) is understood. However, this is not always the case (for instance, in \(p\)-adic groups \(\mathbb{Q}_{p}\), which will be of interest to us).
\end{rem}

\begin{proof}
For a locally finite set \(P\subset G\) and \(x\in G\), let us denote
\[d\left(x,P\right)\coloneqq \inf\left\{d\left(x,p\right): p\in P\right\} \text{ and }D\left(x,P\right)\coloneqq \left\{p\in P: d\left(x,p\right)=d\left(x,P\right)\right\}.\]
Since \(P\) is locally finite, \(D\left(x,P\right)\) is nonempty and finite for all \(x\in G\). Since \(G\) is isomorphic as a standard Borel space to \(\mathbb{R}\), there is a Borel linear order \(\prec\) on \(G\). For a locally finite set \(P\subset G\) and \(p\in P\), define
\[\Theta\left(P,p\right)\coloneqq\big\{x\in G:p\in D\left(x,P\right)\text{ and }xp^{-1}=\min\nolimits_{\prec}xD\left(x,P\right)^{-1}\big\},\]
where \(\min_{\prec}xD\left(x,P\right)^{-1}\) denotes the \(\prec\)-minimal element of the finite set \(xD\left(x,P\right)^{-1}\). Since \(\prec\) is Borel, each \(\Theta\left(P,p\right)\) is Borel, and one can verify that \(\mathscr{T}\left(P\right)\coloneqq \{\Theta\left(P,p\right): p\in P\}\) forms a Borel partition of \(G\). Then Part (1) follows from the construction and, since \(P\) is locally finite, \(\Theta\left(P,p\right)\)'s interior is nonempty.

For Part (2), fix some locally finite sets \(P\subseteq Q\) and \(p\in P\). Let \(x\in\Theta\left(Q,p\right)\) be arbitrary, thus \(p\in D\left(x,Q\right)\) and \(xp^{-1}=\min_{\prec}xD\left(x,Q\right)^{-1}\), and we must show that \(x\in\Theta\left(P,p\right)\), namely: (i) \(p\in D\left(x,P\right)\) and (ii) \(xp^{-1}=\min_{\prec}xD\left(x,P\right)^{-1}\). For (i), since \(p\in D\left(x,Q\right)\) and \(P\subseteq Q\), we have
\[d\left(x,p\right)=d\left(x,Q\right)\leq d\left(x,P\right)\leq d\left(x,p\right),\]
so \(d\left(x,p\right)=d\left(x,P\right)\)
and \(p\in D\left(x,P\right)\). For (ii), note that \(D\left(x,P\right)\subseteq D\left(x,Q\right)\), since for every \(r\in D\left(x,P\right)\), using that \(p\in D\left(x,Q\right)\), it holds that
\[d\left(x,r\right)=d\left(x,P\right)\leq d\left(x,p\right)=d\left(x,Q\right)\leq d\left(x,P\right)=d\left(x,r\right),\]
hence \(d\left(x,r\right)=d\left(x,Q\right)\), so that \(r\in D\left(x,Q\right)\). Then recalling that \(p\in D\left(x,P\right)\) by (i), we deduce that
\[xp^{-1}=\min\nolimits_{\prec}xD\left(x,Q\right)^{-1}\preccurlyeq\min\nolimits_{\prec}xD\left(x,P\right)^{-1}\preccurlyeq xp^{-1},\]
hence \(xp^{-1}=\min_{\prec}xD\left(x,P\right)^{-1}\). This concludes the proof that \(\Theta\left(P,p\right)\supseteq\Theta\left(Q,p\right)\) for all \(p\in P\). In case that \(\Theta\left(P,p\right)=\Theta\left(Q,p\right)\) modulo \(m_{G}\) for all \(p\in P\), we obtain
\[G=\bigsqcup\nolimits_{p\in P}\Theta\left(P,p\right)=\bigsqcup\nolimits_{p\in P}\Theta\left(Q,p\right)\subseteq\bigsqcup\nolimits_{q\in Q}\Theta\left(Q,q\right)=G\text{ modulo }m_{G}.\]
Then for \(q\in Q\setminus P\) we have \(m_{G}\left(\Theta\left(Q,q\right)\right)=0\) while \(\Theta\left(Q,q\right)\)'s interior is nonempty, hence \(P=Q\).

For Part (3), let \(g\in G\) and a locally finite set \(P\subset G\). By the right-invariance of \(d\) we have
\[D\left(x,Pg\right)=D\big(xg^{-1},P\big)g,\quad x\in G.\]
It follows that for every \(p\in P\) and \(x\in G\),
\begin{align*}
x\in\Theta\left(Pg,pg\right)
&\iff pg\in D\left(x,Pg\right)\text{ and }xg^{-1}p^{-1}=\min\nolimits_{\prec}xD\left(x,Pg\right)^{-1}\\
&\iff pg\in D\left(xg^{-1},P\right)g\text{ and }xg^{-1}p^{-1}=\min\nolimits_{\prec}xg^{-1}D\left(xg^{-1},P\right)^{-1}\\
&\iff p\in D\left(xg^{-1},P\right)\text{ and }xg^{-1}p^{-1}=\min\nolimits_{\prec}xg^{-1}D\left(xg^{-1},P\right)^{-1}\\
&\iff xg^{-1}\in\Theta\left(P,p\right)\iff x\in\Theta\left(P,p\right)g.\qedhere
\end{align*}
\end{proof}

\subsection{A higher order Kac's lemma}

Given an lcsc group \(G\), pick some equivariant Voronoi tessellations scheme \(\mathscr{T}:P\mapsto\{\Theta\left(P,p\right): p\in P\}\) using Proposition~\ref{prop:voronoi}. Given a Borel \(G\)-space \(X\) with a cross section \(Y\), every \(x\in X\) is associated the locally finite set \(Y_{x}\subset G\), which in turn produces the Voronoi tesselation
\[\mathscr{T}\left(Y_{x}\right)=\left\{\Theta\left(Y_{x},g\right):g\in Y_{x}\right\}.\]
Note that for every \(y\in Y\) we have \(e_{G}\in Y_{y}\), hence \(\Theta\left(Y_{y},e_{G}\right)\in\mathscr{T}\left(Y_{y}\right)\). Note that the case \(r=1\) of the upcoming Proposition~\ref{prop:Kac} is the formula
\[\mu\left(X\right)=\int_{Y^{\otimes r}}m_{G}\left(\Theta\left(Y_{y},e_{G}\right)\right)d\mu_{Y}\left(y\right).\]

\begin{prop}[Kac's lemma]\label{prop:Kac}
For every transverse \(G\)-space \(\left(X,\mu,Y\right)\), the following formula holds:
\begin{equation}\label{eq:Kac}
I_{\mu}^{r}\left(Y\right)=\int_{Y^{\otimes r}}m_{G}\left(\Theta\left(Y_{y_{1}}\cap\dotsm\cap Y_{y_{r}},e_{G}\right)\right)d\mu_{Y}^{\otimes r}\left(y_{1},\dotsc,y_{r}\right),\quad r\in\mathbb{N},
\end{equation}
for any choice of an equivariant Voronoi tessellation scheme.
\end{prop}

\begin{proof}[Proof of Proposition~\ref{prop:Kac}]
Let us first verify that the integrand function in \eqref{eq:Kac} is measurable. It suffices to show this for \(r=1\) (the case of general \(r\in\mathbb{N}\) follows from the case \(r=1\) by looking at the cross section \(Y^{\otimes r}\)). Using the Lusin-Novikov theorem, pick Borel maps \(\lambda_{n}:Y\to G\), \(n\in\mathbb{N}\), such that \(Y_{y}=\left\{ \lambda_{n}\left(y\right):n\in\mathbb{N}\right\}\) for every \(y\in Y\). Then
\[\theta:G\times Y\longrightarrow\mathbb{R}_{\geq 0},\quad\theta\left(g,y\right)\coloneqq d\left(g,Y_{y}\right)=\inf\left\{ d\left(g,\lambda_{n}\left(y\right)\right):n\in\mathbb{N}\right\},\]
is a Borel function, and hence
\[E\coloneqq\left\{ \left(h,g,y\right):h\in D\left(g,Y_{y}\right)\right\} =\left\{ \left(h,g,y\right):d\left(g,h\right)=\theta\left(g,y\right)\right\} \subset G\times G\times Y\]
is a Borel set. Consequently, since \(D\left(g,Y_{y}\right)=\left\{h:\left(h,g,y\right)\in E\right\} \) is finite for every \(\left(g,y\right)\in G\times Y\), again by the Lusin--Novikov theorem there are Borel maps \(\lambda_{n}^{\prime}:G\times Y\to G\), \(n\in\mathbb{N}\), such that \(\left\{ h:\left(h,g,y\right)\in E\right\} =\left\{ \lambda_{n}^{\prime}\left(g,y\right):n\in\mathbb{N}\right\} \) for every \(\left(g,y\right)\in G\times Y\). All together, we can write
\begin{align*}
S
&\coloneqq\left\{ \left(g,y\right):g\in\Theta\left(Y_{y},e_{G}\right)\right\}\\
&=\left\{ \left(g,y\right):e_{G}\in D\left(g,Y_{y}\right)\right\}\cap\big\{ \left(g,y\right):g=\min\nolimits_{\prec}gD\left(g,Y_{y}\right)^{-1}\big\}\\
&=\left\{ \left(g,y\right):d\left(g,e_{G}\right)=\theta\left(g,y\right)\right\}\cap\bigcap\nolimits_{n\in\mathbb{N}}\big\{ \left(g,y\right):g\preccurlyeq g\lambda_{n}^{\prime}\left(g,y\right)^{-1}\big\},
\end{align*}
and since \(\prec\) is a Borel order, this \(S\) is a Borel set. Then by~\cite[Theorem (17.25)]{kechris2012classical}, the function \(Y\to\mathbb{R}_{\geq 0}\), \(y\mapsto m_{G}\left(S_{y}\right)=m_{G}\left(\Theta\left(Y_{y},e_{G}\right)\right)\), is measurable, as required.

Let us now prove the identity \eqref{eq:Kac}. For a transverse \(G\)-space \(\left(X,\mu,Y\right)\) and an equivariant Voronoi tessellation scheme \(\mathscr{T}\), define the Borel function
\[f:G\times Y^{\otimes r}\longrightarrow\left[0,+\infty\right],\quad f\left(g,\vec{y}\right)\coloneqq\mathbf{1}_{\Theta\small(Y_{\vec{y}}^{\otimes r},e_{G}\small)}\left(g\right).\]
By Campbell theorem~\ref{prop:campbell}, as the transverse \(G\)-space \(\left(Y^{\left[r\right]},\mu^{\left[r\right]},Y^{\otimes r}\right)\) has transverse measure \(\mu_{Y}^{\otimes r}\),
\begin{equation}\label{eq:Kac1}
\int_{Y^{\otimes r}}m_{G}\big(\Theta\big(Y_{\vec{y}}^{\otimes r},e_{G}\big)\big)d\mu_{Y}^{\otimes r}\left(\vec{y}\right)=m_{G}\otimes\mu_{Y}^{\otimes r}\left(f\right)=\mu^{\left[r\right]}\left(f_{Y^{\left[r\right]}}\right).
\end{equation}
For every \(\vec{x}\in Y^{\left[r\right]}\) and \(g\in G\), using that \(Y_{g\diag\vec{x}}^{\otimes r}=Y_{\vec{x}}^{\otimes r}g^{-1}\) and the equivariance in Proposition~\ref{prop:voronoi}(3),
\[\Theta\big(Y_{g\diag\vec{x}}^{\otimes r}\,\,,e_{G}\big)=\Theta\left(Y_{\vec{x}}^{\otimes r}g^{-1},e_{G}\right)=\Theta\left(Y_{x}^{\otimes r},g\right)g^{-1},\]
and therefore the \(Y^{\left[r\right]}\)-periodization of \(f\) is
\begin{align*}
f_{Y^{\left[r\right]}}\left(\vec{x}\right)
&=\sum\nolimits_{g\in Y_{\vec{x}}^{\otimes r}}\mathbf{1}_{\Theta\big(Y_{g\diag\vec{x}}^{\otimes r}\,,e_{G}\big)}\left(g^{-1}\right)\\
&=\sum\nolimits_{g\in Y_{\vec{x}}^{\otimes r}}\mathbf{1}_{\Theta\big(Y_{\vec{x}}^{\otimes r}\,,e_{G}\big)g^{-1}}\left(g^{-1}\right)=\sum\nolimits_{g\in Y_{\vec{x}}^{\otimes r}}\mathbf{1}_{\Theta\left(Y_{\vec{x}}^{\otimes r}\,,g\right)}\left(e_{G}\right)=1,\quad\vec{x}\in Y^{\left[r\right]},
\end{align*}
where the last equality follows since \(\mathscr{T}\left(Y_{\vec{x}}^{\otimes r}\right)=\{\Theta\big(Y_{\vec{x}}^{\otimes r},g\big):g\in Y_{\vec{x}}\}\) is a partition of \(G\). It follows that
\[\mu^{\left[r\right]}\left(f_{Y^{\left[r\right]}}\right)=\mu^{\left[r\right]}\big(Y^{\left[r\right]}\big)=I_{\mu}^{r}\left(Y\right),\]
which together with \eqref{eq:Kac1} completes the proof.
\end{proof}

We can now deduce that the intersection covolume is unchanged under transverse \(G\)-factors.

\begin{prop}\label{prop:interfactor}
Let \(\left(X,\mu,Y\right)\) and \(\left(W,\nu,Z\right)\) be transverse \(G\)-spaces. If there exists a transverse \(G\)-factor \(\left(X,\mu,Y\right)\to\left(W,\nu,Z\right)\) then
\[I_{\mu}^{r}\left(Y\right)=I_{\nu}^{r}\left(Z\right),\quad r\in\mathbb{N}.\]
\end{prop}

\begin{proof}
By Proposition~\ref{prop:conctrans}, if there exists a transverse \(G\)-factor \(\left(X,\mu,Y\right)\to\left(W,\nu,Z\right)\) then there exists such a concrete transverse \(G\)-factor, say \(\phi_{o}\) which is defined on a \(\mu\)-conull \(G\)-invariant Borel set \(X_{o}\subseteq X\) and with a cross section \(Y_{o}\) measure equivalent to \(Y\). By Proposition~\ref{prop:factor} we have \(\nu_{Z}=\phi_{o\ast}\mu_{Y}=\phi_{o\ast}\mu_{Y_{o}}\). Then for every \(r\in\mathbb{N}\) also \(\nu_{Z}^{\otimes r}=\phi_{o\ast}\mu_{Y_{o}}^{\otimes r}\) when applying \(\phi_{o}\) diagonally, and by Kac's lemma~\ref{prop:Kac} we get
\begin{align*}
I_{\nu}^{r}\left(Z\right)	
&=\int_{Z^{\otimes r}}m_{G}\left(\Theta\left(Z_{\vec{z}}^{\otimes r},e_{G}\right)\right)d\phi_{o\ast}\mu_{Y_{o}}^{\otimes r}\left(\vec{z}\right)\\
&=\int_{Y_{o}^{\otimes r}}m_{G}\big(\Theta\big(Z_{\phi_{o}\left(\vec{y}\right)}^{\otimes r},e_{G}\big)\big)d\mu_{Y_{o}}^{\otimes r}\left(\vec{y}\right)\\
&=\int_{Y_{o}^{\otimes r}}m_{G}\big(\Theta\big(Y_{\vec{y}}^{\otimes r},e_{G}\big)\big)d\mu_{Y_{o}}^{\otimes r}\left(\vec{y}\right)\\
&=\int_{Y^{\otimes r}}m_{G}\big(\Theta\big(Y_{\vec{y}}^{\otimes r},e_{G}\big)\big)d\mu_{Y}^{\otimes r}\left(\vec{y}\right)=I_{\mu}^{r}\left(Y\right),
\end{align*}
where in the fourth equality we used that \(Y\) and \(Y_{o}\) are measure equivalent.
\end{proof}

\subsection{Finiteness of the intersection covolume}

The following theorem generalizes Theorem~\ref{mthm:finintcov}.

\begin{thm}\label{thm:finiteiv}
Let \(\left(X,\mu,Y\right)\) be a transverse \(G\)-space and let \(r\in\mathbb{N}\). If \(e_{G}\in\Lambda_{Y}^{2r-1}\) is an isolated point, then \(I_{\mu}^{r}\left(Y\right)<+\infty\).
\end{thm}

Recall that for every cut--and--project scheme \(\left(G,H;\Gamma\right)\) and a window \(W\subset H\), the return times set \(\Lambda_{Y_{W}}\) of \(Y_{W}\) is discrete in \(G\) of any order, that is, \(\Lambda_{Y_{W}}^{r}\) is locally finite (and in particular \(e_{G}\in\Lambda_{Y_{W}}^{r}\) is an isolated point) for all \(r\in\mathbb{N}\) (see~\cite[Proposition 2.13]{BjHa2018}, whose proofs holds for all \(r\in\mathbb{N}\)). Then from Theorem~\ref{thm:finiteiv} it immediately follows that:

\begin{cor}
Let \(\left(G,H;\Gamma\right)\) be a cut--and--project scheme, \(\left(\Xi,\xi\right)\) the associated probability preserving \(G\)-space, and \(Y_{W}\) the cross section for some window \(W\subset H\). Then
\[I_{\xi}^{r}\left(Y_{W}\right)<+\infty\text{ for all }r\in\mathbb{N}.\]
\end{cor}

\begin{proof}[Proof of Theorem~\ref{thm:finiteiv}]
Let \(\Lambda_{Y}\) be the return times set of \(Y\), and note that by the assumption \(\Lambda_{Y}^{2r-1}\), and a fortiori \(\Lambda_{Y}\), is locally finite. For the Borel \(G\)-space \(Y^{\left[r-1\right]}\) (in the diagonal action), define a new cross section \(\widetilde{Y}^{\left(r\right)}\) by
\[\widetilde{Y}^{\left(r\right)}\coloneqq\Lambda_{Y}^{r-1}\diag Y^{\otimes\left(r-1\right)}\subseteq Y^{\left[r-1\right]},\]
and note that its return times sets are
\[\widetilde{Y}_{\vec{z}}^{\left(r\right)}=\Lambda_{Y}^{r-1}Y_{\vec{z}}^{\otimes\left(r-1\right)}=\Lambda_{Y}^{r-1}\left(Y_{z_{2}}\cap\dotsm\cap Y_{z_{r}}\right)\text{ for }\vec{z}=\left(z_{2},\dotsc,z_{r}\right)\in Y^{\left[r-1\right]}.\]
Note also that for every \(\widetilde{y}\in\widetilde{Y}^{\left(r\right)}\), writing \(\widetilde{y}=h\diag\vec{y}\) for some \(h\in\Lambda_{Y}^{r-1}\) and \(\vec{y}\in Y^{\otimes\left(r-1\right)}\), we have
\[\widetilde{Y}_{\widetilde{y}}^{\left(r\right)}=\Lambda_{Y}^{r-1}Y_{\vec{y}}^{\otimes\left(r-1\right)}h^{-1}\subseteq\Lambda_{Y}^{2r-1},\]
since \(Y_{\vec{y}}^{\otimes\left(r-1\right)}\subseteq Y_{y_{1}}\subseteq\Lambda_{Y}\), and therefore
\[\Lambda_{\widetilde{Y}^{\left(r\right)}}\subseteq\Lambda_{Y}^{2r-1}.\]
From the assumption that \(e_{G}\in\Lambda_{Y}^{2r-1}\) is an isolated point it follows that \(\widetilde{Y}^{\left(r\right)}\) is a separated cross section, and therefore, using Lemma~\ref{lem:separatedli}, it is locally integrable with respect to \(\mu^{\left[r-1\right]}\), thus
\begin{equation}\label{eq:Ytilde}
\iota_{\mu^{\left[r-1\right]}}\big(\widetilde{Y}^{\left(r\right)}\big)=\mu_{\widetilde{Y}}^{\left[r-1\right]}\big(\widetilde{Y}^{\left(r\right)}\big)<+\infty.
\end{equation}

Recall the Borel \(G\)-space \(\overline{Y}^{\left[r\right]}\) with its cross section \(Y^{\left[r\right]}\). In order to relate its return times sets to those of \(\widetilde{Y}^{\left[r\right]}\), fix a Borel section \(\lambda:X\to G\) such that \(\lambda\left(x\right)\in Y_{x}\) for every \(x\in X\), and for every \(\left(x,\vec{z}\right)=\left(x,z_{2},\dotsc,z_{r}\right)\in\overline{Y}^{\left[r\right]}\) we find that
\begin{equation}\label{eq:Ytildereturn}
\begin{aligned}
Y_{\left(x,\vec{z}\right)}^{\left[r\right]}	
&=\left(Y_{z_{r}}^{-1}\cap\dotsm\cap Y_{z_{2}}^{-1}\right)Y_{x}=\big(Y_{\vec{z}}^{\otimes\left(r-1\right)}\big)^{-1}Y_{x}=\big(Y_{\vec{z}}^{\otimes\left(r-1\right)}\big)^{-1}Y_{\lambda\left(x\right).x}\lambda\left(x\right)\\
&\subseteq\big(Y_{\vec{z}}^{\otimes\left(r-1\right)}\big)^{-1}\Lambda_{Y}\lambda\left(x\right)=\big(\Lambda_{Y}Y_{\vec{z}}^{\otimes\left(r-1\right)}\big)^{-1}\lambda\left(x\right)=\big(\widetilde{Y}_{\vec{z}}^{\left(r\right)}\big)^{-1}\lambda\left(x\right).
\end{aligned}
\end{equation}
It follows that for every compact set \(K\subset G\),
\begin{align*}
&\int_{\overline{Y}^{\left[r\right]}}\big|Y_{\left(x,\vec{z}\right)}^{\left[r\right]}\cap K\big|d\overline{\mu}^{\left[r\right]}\left(x,\vec{z}\right)\\
\eqref{eq:Ytildereturn}&\qquad \leq\int_{\overline{Y}^{\left[r\right]}}\big|\big(\widetilde{Y}_{\vec{z}}^{\left(r\right)}\big)^{-1}\lambda\left(x\right)\cap K\big|d\overline{\mu}^{\left[r\right]}\left(x,\vec{z}\right)\\
&\qquad=\int_{\overline{Y}^{\left[r\right]}}\big|\widetilde{Y}_{\vec{z}}^{\left(r\right)}\cap\lambda\left(x\right)K^{-1}\big|d\overline{\mu}^{\left[r\right]}\left(x,\vec{z}\right)\\
&\qquad=\int_{X}\Big(\int_{Y^{\left[r-1\right]}}\big|\widetilde{Y}_{\vec{z}}^{\left(r\right)}\cap\lambda\left(x\right)K^{-1}\big|d\mu^{\left[r-1\right]}\left(\vec{z}\right)\Big)d\mu\left(x\right)\\
\left(\ast\right)&\qquad=\int_{X}\Big(m_{G}\big(K\lambda\left(x\right)^{-1}\big)\cdot\mu_{\widetilde{Y}}^{\left[r-1\right]}\big(\widetilde{Y}^{\left(r\right)}\big)\Big)d\mu\left(x\right)\\
\text{(unimodularity and }\eqref{eq:Ytilde})&\qquad=m_{G}\left(K\right)\cdot\mu_{\widetilde{Y}^{\left(r\right)}}^{\left[r-1\right]}\big(\widetilde{Y}^{\left(r\right)}\big)<+\infty,
\end{align*}
where \(\left(\ast\right)\) is by Campbell theorem~\ref{prop:campbell}, applied for each \(x\in X\) to the Borel function
\[f^{x}:G\times\widetilde{Y}^{\left(r\right)}\longrightarrow\left[0,+\infty\right],\quad f^{x}\left(g,\widetilde{y}\right)=\mathbf{1}_{K\lambda\left(x\right)^{-1}}\left(g\right),\]
whose \(Y^{\left[r-1\right]}\)-periodization (as in the proof of Proposition~\ref{prop:intensityli}) is
\[f_{Y^{\left[r-1\right]}}^{x}\left(\vec{z}\right)=\big|\widetilde{Y}_{\vec{z}}^{\left(r\right)}\cap\lambda\left(x\right)K^{-1}\big|.\]
This computation shows that \(Y^{\left[r\right]}\) is a locally integrable cross section for \(\big(\overline{Y}^{\left[r\right]},\overline{\mu}^{\left[r\right]}\big)\), which means that \(\overline{\mu}_{Y^{\left[r\right]}}^{\left[r\right]}\left(Y^{\left[r\right]}\right)<+\infty\) (recall Proposition~\ref{prop:intensityli}). Finally, in order to apply Proposition~\ref{prop:intmeasiden}, note that from \eqref{eq:Ytildereturn} for \(r=2\), for all \(\left(x,z\right)\in X^{\otimes 2}=\overline{Y}^{\left[2\right]}\) we have
\[Y_{x}^{-1}Y_{z}\subseteq\big(\widetilde{Y}_{z}^{\left(2\right)}\big)^{-1}\lambda\left(x\right),\]
and since \(\widetilde{Y}^{\left(2\right)}\) is locally integrable (\(e_{G}\in\Lambda_{Y}^{3}\subseteq\Lambda_{Y}^{2r-1}\) is an isolated point), \(Y_{x}^{-1}Y_{z}\) is a locally finite set, thus \eqref{eq:doubs} is satisfied. Then from Proposition~\ref{prop:intmeasiden} it follows that
\[I_{\mu}^{r}\left(Y\right)=\mu^{\left[r\right]}\big(Y^{\left[r\right]}\big)=\overline{\mu}_{Y^{\left[r\right]}}^{\left[r\right]}\big(Y^{\left[r\right]}\big)<+\infty.\qedhere\]
\end{proof}

\subsection{Intersection space for nonuniform lattices}\label{sct:nonunif}

Recall that in point processes, the canonical cross section is typically not separated, though it is often locally integrable. The following example, originally presented in~\cite[Example 7.3]{bjorklund2025int}, illustrates a different kind of non-separated transverse \(G\)-space: one that arises as an intersection space. This example further shows that even when a transverse \(G\)-space \(\left(X,\mu,Y\right)\) is infinitely separated, the corresponding intersection space \(\left(Y^{\left[r\right]},\mu^{\left[r\right]},Y^{\otimes r}\right)\), while locally integrable by Theorem~\ref{thm:finiteiv}, may nevertheless fail to be at all separated.

For \(d\in\mathbb{N}\), consider the nonuniform lattice
\[\Gamma\coloneqq\mathrm{SL}_{d}\left(\mathbb{Z}\right)<G\coloneqq\mathrm{SL}_{d}\left(\mathbb{R}\right),\]
and the associated homogeneous space \(\Gamma\backslash G\) with the one-point cross section \(\{\Gamma\}\). Consider the Borel \(G\)-space \(X\coloneqq \left(\Gamma\backslash G\right)^{\otimes 2}\) with the action of \(G\) on the first coordinate, and the intersection space \(Y\coloneqq \{\Gamma\}^{\left[2\right]}\) as a cross section for \(X\). Thus,
\[Y\coloneqq\left\{ \left(\Gamma g,\Gamma g\right):g\in G \right\} \subset X,\]
and a direct computations shows that
\[Y_{x}=g_{1}^{-1}\Gamma g_{2}\times\left\{e_{G}\right\} \text{ for }x=\left(\Gamma g_{1},\Gamma g_{2}\right)\in X,\]
so that \(Y\) is indeed a cross section for the Borel \(G\)-space \(X\). Note that when \(y=\left(\Gamma g,\Gamma g\right)\in Y\) we have \(Y_{y}=\Gamma^{g}\), and this shows that \(Y\) is not separated. Indeed,
\[\Lambda_{Y}=\bigcup\nolimits_{y\in Y}Y_{y}=\bigcup\nolimits_{g\in G}\Gamma ^{g}=\mathrm{SL}_{d}\left(\mathbb{Z}\right)^{\mathrm{SL}_{d}\left(\mathbb{R}\right)},\]
and conjugations of unipotent elements accumulate at the identity, and thus \(e_{G}\in\Lambda_{Y}\) is not isolated.

Nevertheless, \(Y\) is locally integrable with respect to the canonical \(G\times G\)- hence \(G\)-invariant measure \(m_{X}=m_{\Gamma\backslash G}\otimes m_{\Gamma\backslash G}\) on \(X\). This is a result of the general Theorem~\ref{thm:finiteiv}, since \(\{\Gamma\}\) is infinitely separated cross section for \(\Gamma\backslash G\), but we can also see it here directly. Indeed, let \(K\subset G\) be a compact set, and for every \(x=\left(\Gamma g_{1},\Gamma g_{2}\right)\in X\) we have
\[\left|Y_{x}\cap K\right|=\left|g_{1}^{-1}\Gamma g_{2}\cap K\right|=\left|g_{1}^{-1}\Gamma \cap Kg_{2}^{-1}\right|.\]
For an arbitrary fixed \(g_{2}\in G\), define the Borel function
\[f_{K,g_{2}}:G \longrightarrow\left[0,+\infty\right],\quad f_{K,g_{2}}\left(g\right)\coloneqq \mathbf{1}_{Kg_{2}^{-1}}\left(g^{-1}\right).\]
Note that for every \(g\in G\) we have
\[\sum\nolimits_{\gamma\in\Gamma }f_{K,g_{2}}\left(\gamma g_{1}\right)=\sum\nolimits_{\gamma\in\Gamma }\mathbf{1}_{Kg_{2}^{-1}}\left(g_{1}^{-1}\gamma^{-1}\right)=\left|g_{1}^{-1}\Gamma\cap Kg_{2}^{-1}\right|,\]
and it follows that
\[\int_{\Gamma\backslash G}\left|g_{1}^{-1}\Gamma \cap Kg_{2}^{-1}\right|dm_{\Gamma\backslash G}\left(\Gamma g_{1}\right)=\int_{G }f_{K,g_{2}}\left(g\right)dm_{G }\left(g\right)=m_{G }\left(Kg_{2}^{-1}\right)=m_{G}\left(K\right),\]
independently of \(g_{2}\). We obtain by unimodularity and Fubini's theorem that
\begin{align*}
\int_{X}\left|Y_{x}\cap K\right|dm_{X}\left(x\right)
&=\iint\nolimits_{\left(\Gamma\backslash G\right)^{\otimes 2}}\left|g_{1}^{-1}\Gamma\cap Kg_{2}^{-1}\right|dm_{\Gamma\backslash G}^{\otimes 2}\left(\Gamma g_{1},\Gamma g_{2}\right)\\
& =\int_{\Gamma\backslash G}m_{G_{o}}\left(K\right)dm_{\Gamma\backslash G}\left(\Gamma g_{2}\right)=m_{G}\left(K\right)<+\infty.
\end{align*}

\section{The basic inequality and its extremes}

The following inequality, appeared in Theorem~\ref{thm:mthm}, is a simple consequence of Kac's lemma~\ref{prop:Kac}.

\begin{prop}[The Basic Inequality]\label{prop:basicineq}
For every transverse \(G\)-space \(\left(X,\mu,Y\right)\),
\begin{equation}\label{eq:basicineq}
I_{\mu}^{r+1}\left(Y\right)\geq\iota_{\mu}\left(Y\right)\cdot I_{\mu}^{r}\left(Y\right),\quad r\in\mathbb{N}.
\end{equation}
In particular,
\[I_{\mu}^{r+1}\left(Y\right)\geq\iota_{\mu}\left(Y\right)^{r},\quad r\in\mathbb{N}.\]
\end{prop}

The case \(r=1\) in Proposition~\ref{prop:basicineq} is the inequality \(I_{\mu}\left(Y\right)\geq\iota_{\mu}\left(Y\right)\) appears in Theorem~\ref{thm:mthm}.

\begin{proof}
Let \(r\in\mathbb{N}\) be arbitrary. For every \(\left(y_{1},\dotsc,y_{r},y_{r+1}\right)\in Y^{\otimes r+1}\) we have
\[Y_{\left(y_{1},\dotsc,y_{r},y_{r+1}\right)}^{\otimes r+1}\subseteq Y_{\left(y_{1},\dotsc,y_{r}\right)}^{\otimes r},\]
so by Proposition~\ref{prop:voronoi}(2), the corresponding Voronoi tessellation satisfies
\[\Theta\big(Y_{\left(y_{1},\dotsc,y_{r},y_{r+1}\right)}^{\otimes r+1},e_{G}\big)\supseteq \Theta\big(Y_{\left(y_{1},\dotsc,y_{r}\right)}^{\otimes r},e_{G}\big).\]
Then \eqref{eq:basicineq} follows using Kac's lemma~\ref{prop:Kac}, and for a later use we record the explicit inequality:
\begin{equation}\label{eq:latus}
\begin{aligned}
I_{\mu}^{r+1}\left(Y\right)
&=\int_{Y^{\otimes r+1}}m_{G}\big(\Theta\big(Y_{\left(y_{1},\dotsc,y_{r},y_{r+1}\right)}^{\otimes r+1},e_{G}\big)\big)d\mu_{Y}^{\otimes r+1}\left(y_{1},\dotsc,y_{r},y_{r+1}\right)\\
&\geq\int_{Y^{\otimes r+1}}m_{G}\big(\Theta\big(Y_{\left(y_{1},\dotsc,y_{r}\right)}^{\otimes r},e_{G}\big)\big)d\mu_{Y}^{\otimes r+1}\left(y_{1},\dotsc,y_{r},y_{r+1}\right)\\
&=\int_{Y}d\mu_{Y}\left(y_{r+1}\right)\cdot \int_{Y^{\otimes r}}m_{G}\big(\Theta\big(Y_{\left(y_{1},\dotsc,y_{r}\right)}^{\otimes r},e_{G}\big)\big)d\mu_{Y}^{\otimes r}\left(y_{1},\dotsc,y_{r}\right)\\
&=\mu_{Y}\left(Y\right)\cdot I_{\mu}^{r}\left(Y\right)=\iota_{\mu}\left(Y\right)\cdot I_{\mu}^{r}\left(Y\right).
\end{aligned}
\end{equation}
Finally, inducting on \(r\) and using that \(I_{\mu}^{1}\left(Y\right)=\mu\left(X\right)=1\), we get \(I_{\mu}^{r+1}\left(Y\right)\geq\iota_{\mu}\left(Y\right)^{r}\) for every \(r\in\mathbb{N}\).
\end{proof}

Let us now consider the totally periodic case, where the intersection covolume is minimal. Let \(\Gamma<G\) be a lattice. There is associated a canonical transverse \(G\)-space,
\[\left(\Gamma\backslash G,m_{\Gamma\backslash G},\{\Gamma\}\right),\]
with the action given by \(g_{o}.\Gamma g=\Gamma gg_{o}^{-1}\) and \(m_{\Gamma\backslash G}\) is the unique \(G\)-invariant probability measure on \(\Gamma\backslash G\). Since \(G\) acts transitively on \(\Gamma\backslash G\), the one-point set \(\{\Gamma\}\subseteq \Gamma\backslash G\) is a cross section, with return times sets
\[\{\Gamma\}_{\Gamma g}=\Gamma g,\quad \Gamma g\in\Gamma\backslash G.\]
Using Kac's lemma~\ref{prop:Kac}, we obtain
\begin{align*}
I_{m_{\Gamma\backslash G}}^{r}\left(\left\{\Gamma\right\}\right)
&=\mu_{\left\{\Gamma\right\}}\left(\left\{\Gamma\right\} \right)^{r}\cdot m_{G}\left(\Theta\left(\Gamma,e_{G}\right)\right)\\
&=\mu_{\left\{\Gamma\right\}}\left(\left\{\Gamma\right\} \right)^{r}\cdot m_{\Gamma\backslash G}\left(\Gamma\backslash G\right)=\mu_{\left\{\Gamma\right\}}\left(\left\{\Gamma\right\} \right)^{r}=\iota_{m_{\Gamma\backslash G}}\left(\{\Gamma\}\right)^{r},\qquad r\in\mathbb{N}.
\end{align*}

Next we characterize the extremes of the Basic Inequality \eqref{eq:basicineq}, generalizing Theorem~\ref{thm:mthm}.

\begin{thm}\label{thm:mthm+}
For a transverse \(G\)-space \(\left(X,\mu,Y\right)\) the following are equivalent:
\begin{enumerate}
    \item \(I_{\mu}^{r+1}\left(Y\right)=\iota_{\mu}\left(Y\right)^{r}\) for every \(r\in\mathbb{N}\).
    \item \(I_{\mu}^{r+1}\left(Y\right)=\iota_{\mu}\left(Y\right)^{r}\) for some \(r\in\mathbb{N}\).
    \item There is a lattice \(\Gamma <G\) as well as a transverse \(G\)-factor \(\left(X,\mu,Y\right)\to\left(\Gamma\backslash G,m_{\Gamma\backslash G},\{\Gamma\}\right)\).
\end{enumerate}
\end{thm}

\begin{rem}
When \(\left(X,\mu,Y\right)\) is ergodic with a \(G\)-factor to \(\left(\Gamma\backslash G,\pi_{\ast}\mu\right)\), by~\cite[Theorem 2.5]{zimmer1978induced} it is an \emph{induced action} from \(\Gamma\): there is a probability preserving \(\Gamma\)-space \(\left(Z,\zeta\right)\) and
\[\left(X,\mu,Y\right)\cong\left(\Gamma\backslash G\times Z,m_{\Gamma\backslash G}\otimes\zeta,\{\Gamma\}\times Z\right),\]
with the action given by \(g_{o}.\left(\Gamma g,z\right)=\left(\Gamma gg_{o}^{-1},\lambda\left(g_{o},\Gamma g\right).z\right)\) for a Borel section \(\lambda:G\times\Gamma\backslash G\to\Gamma\)~\cite[II]{zimmer1978induced}.
\end{rem}

\begin{proof}
It is clear that (1) implies (2). To see that (3) implies (1) recall that the intersection covolume of \(\left(\Gamma\backslash G,m_{\Gamma\backslash G},\{\Gamma\}\right)\) is \(I_{m_{\Gamma\backslash G}}^{r}\left(\left\{\Gamma\right\} \right)=m_{\Gamma\backslash G}\left(\Gamma\backslash G\right)^{r}\), so by Proposition~\ref{prop:interfactor} the same is true for (3). We then show that (2) implies (3). Suppose (2) holds for \(r\), and using the Basic Inequality \eqref{eq:basicineq} inductively,
\[I_{\mu}^{r+1}\left(Y\right)\geq\iota_{\mu}\left(Y\right)^{r-1}\cdot I_{\mu}^{2}\left(Y\right)\geq\iota_{\mu}\left(Y\right)^{r}=I_{\mu}^{r+1}\left(Y\right),\]
thus those inequalities are in fact equalities, and in particular
\[I_{\mu}^{2}\left(Y\right)=\iota_{\mu}\left(Y\right).\]
By the computation \eqref{eq:latus} with \(r=1\) and using \(I_{\mu}^{1}\left(Y\right)=1\), this yields
\[\int_{Y^{\otimes2}}m_{G}\big(\Theta\big(Y_{\left(y_{1},y_{2}\right)}^{\otimes2},e_{G}\big)\big)d\mu_{Y}^{\otimes2}\left(y_{1},y_{2}\right)=\int_{Y^{\otimes2}}m_{G}\left(\Theta\left(Y_{y_{1}},e_{G}\right)\right)d\mu_{Y}^{\otimes2}\left(y_{1},y_{2}\right),\]
while for all \(\left(y_{1},y_{2}\right)\in Y^{\otimes 2}\) we simply have
\[m_{G}\big(\Theta\big(Y_{\left(y_{1},y_{2}\right)}^{\otimes2},e_{G}\big)\big)=m_{G}\left(\Theta\left(Y_{y_{1}}\cap Y_{y_{2}},e_{G}\right)\right)\geq m_{G}\left(\Theta\left(Y_{y_{1}},e_{G}\right)\right).\]
Therefore, the Borel set
\[E_{o}\coloneqq \left\{\left(y_{1},y_{2}\right)\in Y^{\otimes 2}: m_{G}\left(\Theta\left(Y_{y_{1}}\cap Y_{y_{2}},e\right)\right)=m_{G}\left(\Theta\left(Y_{y_{1}},e\right)\right)\right\},\]
is \(\mu_{Y}^{\otimes 2}\)-conull. Put the Borel set
\[E_{oo}\coloneqq \left\{\left(y_{1},y_{2}\right)\in Y^{\otimes 2}:\forall g\in Y_{y_{1}}\cap Y_{y_{2}},\,\,g.\left(y_{1},y_{2}\right)\in E_{o}\right\}\subseteq E_{o},\]
and by Proposition~\ref{prop:null}(2), \(E_{oo}\) is still \(\mu_{Y}^{\otimes 2}\)-conull. For all \(y\in Y\) and \(g\in Y_{y}\), by Proposition~\ref{prop:voronoi}(3) and the invariance of \(m_{G}\), we have
\[m_{G}\left(\Theta\left(Y_{g.y},e\right)\right)=m_{G}\small(\Theta\small(Y_{y}g^{-1},e\small)\small)=m_{G}\small(\Theta\left(Y_{y},g\right)g^{-1}\small)=m_{G}\left(\Theta\left(Y_{y},g\right)\right),\]
and similarly, for all \(\left(y_{1},y_{2}\right)\in Y\times Y\) and \(g\in Y_{y_{1}}\cap Y_{y_{2}}\), we have
\[m_{G}\left(\Theta\left(Y_{g.y_{1}}\cap Y_{g.y_{2}},e\right)\right)=m_{G}\left(\Theta\left(Y_{y_{1}}\cap Y_{y_{2}},g\right)\right).\]
This means that for all \(\left(y_{1},y_{2}\right)\in E_{oo}\), the two Voronoi tessellations
\[\left\{\Theta\left(Y_{y_{1}}\cap Y_{y_{2}},g\right): g\in Y_{y_{1}}\cap Y_{y_{2}}\right\} \text{ and }\left\{\Theta\left(Y_{y_{1}},g\right): g\in Y_{y_{1}}\right\},\]
satisfy that \(\Theta\left(Y_{y_{1}}\cap Y_{y_{2}},g\right)=\Theta\left(Y_{y_{1}},g\right)\) modulo \(m_{G}\) for all \(g\in Y_{y_{1}}\cap Y_{y_{2}}\), and by Proposition~\ref{prop:voronoi}(2) we deduce that \(Y_{y_{1}}\cap Y_{y_{2}}=Y_{y_{1}}\). Thus, we have found that
\[Y_{y_{1}}\cap Y_{y_{2}}=Y_{y_{2}}\text{ for }\mu_{Y}^{\otimes2}\text{-a.e. }\left(y_{1},y_{2}\right)\in Y^{\otimes2}.\]
Using the same reasoning with the opposite order of \(y_{1},y_{2}\), we deduce that
\[Y_{y_{1}}=Y_{y_{1}}\cap Y_{y_{2}}=Y_{y_{2}}\text{ for }\mu_{Y}^{\otimes 2}\text{-a.e. }\left(y_{1},y_{2}\right)\in Y^{\otimes 2}.\]
It follows from Fubini's theorem that the set
\[B\coloneqq\{y\in Y:Y_{y}=\Gamma\}\text{ is }\mu_{Y}\text{-conull}.\]
Note that \(\Gamma\) satisfies
\[\Gamma g^{-1}=Y_{y}g^{-1}=Y_{g.y}\text{ for all }y\in B\text{ and }g\in Y_{y}=\Gamma.\]
However, by Proposition~\ref{prop:null}(2) we have \(g.y\in B\) for \(\mu_{Y}\)-a.e. \(y\in B\) and every \(g\in Y_{y}\), and hence
\[\Gamma=\Gamma g^{-1}\text{ for all }g\in\Gamma.\]
We have thus found that \(\Gamma<G\) is a subgroup, and it is discrete since \(Y_{y}\) is locally finite for each \(y\in B\).\footnote{The terms \emph{discrete} and \emph{locally finite} are synonymous, and the distinction in terminology is only contextual.} Pick a Borel map \(\lambda:G.Y\to G\) satisfying \(\lambda\left(x\right)\in Y_{x}\) for \(x\in G.Y\) (using Lusin--Novikov theorem~\ref{thm:lusinnovi}), and define
\[\pi:G.Y\to\Gamma\backslash G,\quad\pi\left(x\right)\coloneqq \Gamma\lambda\left(x\right).\]
By Proposition~\ref{prop:null}(2) we have \(\lambda\left(x\right).x\in B\) for \(\mu\)-a.e. \(x\in X\), and therefore
\[\pi\left(x\right)=\Gamma\lambda\left(x\right)=Y_{\lambda\left(x\right).x}\lambda\left(x\right)=Y_{x}\,\,\text{ for }\mu\text{-a.e. }x\in X.\]
Since the \(G\)-invariant set \(A\coloneqq\{x\in X:Y_{x}=B_{x}\}\) is \(\mu\)-conull, it follows that for every \(g\in G\),
\[\pi\left(g.x\right)=Y_{g.x}=Y_{x}g^{-1}=\pi\left(x\right)g^{-1}\,\,\text{ for }\mu\text{-a.e. }x\in X.\]
Therefore, \(\pi\) forms a \(G\)-factor and, since \(\pi_{\ast}\mu\) is \(G\)-invariant, \(\Gamma\) is a lattice in \(G\) and \(\pi_{\ast}\mu=m_{\Gamma\backslash G}\). Finally, to verify that \(\pi:\left(X,\mu,Y\right)\to\left(\Gamma\backslash G,m_{\Gamma\backslash G},\{\Gamma\}\right)\) forms a a transverse \(G\)-factor, note that
\[Y_{x}=Y_{\lambda\left(x\right).x}\lambda\left(x\right)=\Gamma\lambda\left(x\right)=\left\{ \Gamma\right\} _{\Gamma\lambda\left(x\right)}=\left\{ \Gamma\right\} _{\pi\left(x\right)}\text{ for }\mu\text{-a.e. }x\in X.\qedhere\]
\end{proof}

\subsection{No gap in the basic inequality}\label{sct:nogap}

We demonstrate that, at least for \(G=\mathbb{R}\), there is no gap in the basic inequality \eqref{eq:basicineq}, that is, even for transverse \(\mathbb{R}\)-spaces which are not completely periodic, the intersection covolume can be arbitrarily closed to the intensity. We thus find for every \(\epsilon>0\) a transverse \(\mathbb{R}\)-space \(\left(X,\mu,Y\right)\), and in fact \(Y\) can be chosen to be infinitely separated, such that
\begin{equation}\label{eq:suspflow}
\iota_{\mu}\left(Y\right)<I_{\mu}\left(Y\right)<\left(1+\epsilon\right)\cdot\iota_{\mu}\left(Y\right).
\end{equation}

\bigskip

\begin{wrapfigure}[20]{r}{0.29\textwidth}
  \centering
  \vspace{-\baselineskip}
  \begin{tikzpicture}[scale=2, axis/.style={-Latex, thick}, font=\small]
  \fill[black!10] (0,0) rectangle (1-0.1, 1);
  \fill[black!10] (1-0.1, 0) rectangle (1, 2);
  \draw[black, thick] (0,0) -- (1,0) -- (1,2) -- (1-0.1, 2) -- (1-0.1, 1) -- (0,1) -- cycle;
  \draw[axis] (-0.1,0) -- (1.3,0) node[below left=2pt] {\(z\)};
  \draw[axis] (0,-0.1) -- (0,2.4) node[below left=2pt] {\(t\)};
  \draw[black, very thick] (0,1) -- (1-0.1,1) node[font=\footnotesize, midway, above=0pt] {\(r=1\)};
  \draw[black, very thick] (1-0.1,2) -- (1,2) node[font=\footnotesize, midway, above=0pt] {\(r=2\)};
  \node[font=\large, below=4pt] at (0.4, 0) {\(Y\)};
  \node[font=\footnotesize, below left=2pt] at (0,0) {\(0\)};
  \draw (1-0.1, 2pt) -- (1-0.1, -2pt) node[font=\footnotesize, below=5pt, left=-5pt] {\(0.9\)};
  \draw (1, 2pt) -- (1, -2pt) node[font=\footnotesize, below=5pt, right=-5pt] {\(1\)};
  \draw (2pt, 1) -- (-2pt, 1) node[font=\footnotesize, left] {\(1\)};
  \draw (2pt, 2) -- (-2pt, 2) node[font=\footnotesize, left] {\(2\)};
  \node[font=\large] at (0.5, 0.6) {\(X\)};
  \draw[dashed, gray] (1-0.1, 0) -- (1-0.1, 2);
  \draw[dashed, gray] (0, 1) -- (1-0.1, 1);
  \draw[dashed, gray] (0, 2) -- (1-0.1, 2);
  \end{tikzpicture}
    \captionsetup{labelformat=empty, font=scriptsize}
    \caption{The suspension flow on \(Z=\left[0,1\right]\) under the function \(r=2\) on \(Z_{\epsilon}=\left[0.9,1\right]\) (for \(\epsilon=0.1\)) and \(r=1\) elsewhere, and the cross section \(Y=\left[0,1\right]\times\{0\}\).}
\end{wrapfigure}

Let \(\left(Z,\zeta\right)\) be a nonatomic standard probability space and \(0<\epsilon<1\). Pick a Borel set \(Z_{\epsilon}\subset Z\) with \(\zeta\left(Z_{\epsilon}\right)=\epsilon\), and let
\[r:Z\to\mathbb{R},\quad r\left(z\right)\coloneqq 1+\mathbf{1}_{Z_{\epsilon}}\left(z\right).\]
Let \(T\) be any ergodic probability preserving invertible transformation of \(\left(Z,\zeta\right)\), and define \(\left(X,\mu\right)\) to be the (ergodic) suspension flow under the function \(r\) induced from \(T\), namely the vertical flow defined by \(T\) on
\[X=\{\left(z,t\right)\in Z\times\mathbb{R}_{\geq 0}:0\leq t<r\left(z\right)\},\]
with the \(\mathbb{R}\)-invariant probability measure
\[\mu\coloneqq\frac{1}{1+\epsilon}\cdot\left(\zeta\otimes m_{\mathbb{R}}\right)\mid_{X}.\]
(See~\cite[Ch.~11]{cornfeld2012ergodic} for details on the construction of suspension flows). Consider the separated cross section
\[Y\coloneqq Z\times\left\{0\right\}\subset X.\]
We claim that the ergodic transverse \(\mathbb{R}\)-space \(\left(X,\mu,Y\right)\) satisfies \eqref{eq:suspflow}. 

\bigskip

First, note that the return times sets \(Y_{\left(z,t\right)}\), \(\left(z,t\right)\in X\), are not all cosets of the same lattice, and therefore \(\left(X,\mu,Y\right)\) is not completely periodic. Then the first (strict) inequality in \eqref{eq:suspflow} follows from Theorem~\ref{thm:mthm}. To verify the second inequality in \eqref{eq:suspflow}, start by noting that since the cross section \(Y\) is \(\left(-a,a\right)\)-separated for any \(0<a<1/2\), we have
\begin{align*}
\iota_{\mu}\left(Y\right)=\mu_{Y}\left(Y\right)
&=\frac{1}{2a}\cdot m_{\mathbb{R}}\left(\left(-a,a\right)\right)\cdot\mu_{Y}\left(Y\right)=\frac{1}{2a}\cdot\mu\left(\left(-a,a\right).Y\right)\\
&=\frac{1}{2a}\cdot\mu\left(Z\times\left(-a,a\right)\right)=\frac{1}{2a}\cdot\frac{1}{1+\epsilon}\cdot\zeta\left(Z\right)\cdot m_{\mathbb{R}}\left(\left(-a,a\right)\right)=\frac{1}{1+\epsilon}.
\end{align*}
In order to bound \(I_{\mu}\left(Y\right)\), let us look at another cross section for \(\left(X,\mu\right)\),
\[\widetilde{Y}\coloneqq Y\cup\left(Z_{\epsilon}\times\left\{ 1\right\} \right).\]
One can see that  the return times sets \(\widetilde{Y}_{\left(z,t\right)}\), \(\left(z,t\right)\in X\), are all cosets of the lattice \(\mathbb{Z}<\mathbb{R}\), and thus \(\widetilde{Y}\) is completely periodic. Then by Theorem~\ref{thm:mthm} we obtain
\begin{align*}
I_{\mu}\big(\widetilde{Y}\big)=\iota_{\mu}\big(\widetilde{Y}\big)=\mu_{\widetilde{Y}}\big(\widetilde{Y}\big)
&=\frac{1}{2a}\cdot m_{\mathbb{R}}\left(\left(-a,a\right)\right)\cdot\mu_{\widetilde{Y}}\big(\widetilde{Y}\big)\\
&=\frac{1}{2a}\cdot\mu\big(\left(-a,a\right).\widetilde{Y}\big)=1<\frac{1}{1-\epsilon}=\frac{1+\epsilon}{1-\epsilon}\cdot\iota_{\mu}\left(Y\right).
\end{align*}
Finally, since both \(Y\) and \(\widetilde{Y}\) are separated and \(Y\subset\widetilde{Y}\), by Proposition~\ref{prop:mono} we get
\[I_{\mu}\left(Y\right)\leq I_{\mu}\big(\widetilde{Y}\big)<\frac{1+\epsilon}{1-\epsilon}\cdot\iota_{\mu}\left(Y\right).\]
Since \(\epsilon>0\) is arbitrary, this establishes also the second inequality in \eqref{eq:suspflow}.

\bibliographystyle{amsplain}
\bibliography{references}

\end{document}